\documentclass{amsart}

\usepackage[utf8]{inputenc}
\usepackage[hidelinks]{hyperref}

\usepackage{amsmath}
\usepackage{amssymb}

\usepackage{amsthm}
\theoremstyle{plain}
\newtheorem{theorem}{Theorem}[section]
\newtheorem{lemma}[theorem]{Lemma}
\newtheorem{proposition}[theorem]{Proposition}
\newtheorem{corollary}[theorem]{Corollary}
\theoremstyle{definition}
\newtheorem{definition}[theorem]{Definition}

\newtheorem{example}[theorem]{Example}
\theoremstyle{remark}
\newtheorem{remark}[theorem]{Remark}

\usepackage{tikz-cd}

\title[Hilbert schemes of elliptic surfaces]{Hilbert schemes of elliptic surfaces: group actions and derived categories}
\author{David Zhiyuan Bai}
\date{\today}
\address{Department of Mathematics, Yale University\\
New Haven, CT 06520-8283\\
USA}
\email{david.bai@yale.edu}


\begin{document}

\begin{abstract}
    Let $X\to C$ be an elliptic surface with integral fibers and a section.
    The Hilbert scheme $X^{[n]}$ fibers over $C^{[n]}$.
    We construct a commutative group scheme over the entire base $C^{[n]}$ that embeds as an open subscheme of the Hilbert scheme, such that its action on itself extends to the entirety of $X^{[n]}$.
    We show that the action is $\delta$-regular in the sense of Ng\^o.
    Using the derived McKay correspondence, we construct an exact autoequivalence of $D^b\operatorname{Coh}(X^{[n]})$ whose kernel is a maximal Cohen-Macaulay sheaf on the fiber product.
    We show that this Fourier-Mukai transform intertwines with our group action, i.e.~theorem of the square holds.
    We also discuss the case without a section using the theory of Tate-Shafarevich twists.
\end{abstract}
    \maketitle
    \tableofcontents
    \section{Introduction}
    We work over $k=\mathbb C$.

    Suppose $\mathcal C\to B$ is a family of smooth projective curves, the classical construction of the Jacobian gives rise to a smooth fibration $J_{\mathcal C}\to B$.

    It comes equipped with a package of three properties:
    \begin{enumerate}
        \item (Liouville theorem) The fibration $J_{\mathcal C}\to B$ admits the structure of an abelian scheme.
        Denote its group law by $\mu:J_{\mathcal C}\times_BJ_{\mathcal  C}\to J_{\mathcal C}$.
        \item (Autoduality) There is a Poincar\'e line bundle $\mathcal P$ on $J_{\mathcal C}\times_B J_{\mathcal C}$, invariant under swapping the factors, that induces an exact autoequivalence of $D^b\operatorname{Coh}(J_{\mathcal C})$.
        \item (Theorem of the square) The identity
        \[(\mu\times\operatorname{id})^\ast\mathcal P=\operatorname{pr}_{13}^\ast\mathcal P\otimes\operatorname{pr}_{23}^\ast\mathcal P\]
        holds on $J_{\mathcal C}\times_BJ_{\mathcal C}\times_BJ_{\mathcal C}$, where $\operatorname{pr}_{ij}:J_{\mathcal C}\times_BJ_{\mathcal C}\times_BJ_{\mathcal C}\to J_{\mathcal C}\times_BJ_{\mathcal C}$ is the projection to the $(i,j)$-th factor.
    \end{enumerate}
    A natural question to ask is, then, whether the same package of properties can be found more generally in (not necessarily smooth) fibrations $M\to B$ with a section, such that the general fiber is a principally polarized abelian variety.
    Specifically, one seeks the following:
    \begin{enumerate}
        \item (Liouville theorem) There is an open $M^\circ$ containing the section, and a map $\mu:M^\circ\times_BM\to M$, such that $\mu$ restricts to a commutative group law $M^\circ \times_BM^\circ\to M^\circ$ under which $\mu$ becomes a group scheme action extending the regular action.
        \item (Autoduality) There is a Poincar\'e complex $\mathcal P\in D^b\operatorname{Coh}(M\times_BM)$, invariant under swapping the factors, that induces an exact autoequivalence of $D^b\operatorname{Coh}(M)$.
        \item (Theorem of the square) The restriction $\mathcal P^\circ =\mathcal P|_{M^\circ\times_BM}$ is a line bundle, and
        \[(\mu\times\operatorname{id})^\ast\mathcal P=\operatorname{pr}_{13}^\ast\mathcal P^\circ\otimes\operatorname{pr}_{23}^\ast\mathcal P\]
        holds on $M^\circ\times_BM\times_BM$.
    \end{enumerate}
    Numerous progress has been made on constructing this package in the cases when the fibers are integral.
    The most success has been found when $M$ is a fine compactified Jacobian of a family of reduced curves with planar singularities, where the entire package is now known.
    When the curves are integral, this was established by the seminal work of Arinkin \cite{arinkin}, which was later generalized by Melo, Rapagnetta, and Viviani \cite{mrv1} \cite{mrv2}.

    The properties become much harder when $M$ is not explicitly a moduli space of sheaves compactifying a suitable open of the Picard space of a family of (suitably well-behaved) curves over $B$.
    In this generality, it is not clear where the Liouville theorem should come from.
    When $M\to B$ is a degenerate abelian scheme (e.g.~an algebraic integrable system with integral fibers), this has been established by Arinkin and Fedorov \cite{arinkin-fedorov}.
    In the same paper, they showed a partial autoduality result, but full autoduality remains wide open.

    Few have considered situations where the integrality condition on the fibers is dropped.
    In the case of a projective Lagrangian fibration without any nowhere-reduced fiber (which is automatic when the fibration has a section), Kim \cite[\S4]{kim_neron} recently used the theory of Hodge modules to construct a group scheme action on $M\to B$ that restricts to a free action on the smooth part $M^{\rm sm}\to B$ of the fibration.
    We have been informed of an upcoming follow-up work by de Cataldo, Kim, and Schnell.
    For Hitchin fibrations, a group action is described by de Cataldo, Fringuelli, Fernandez Herrero, and Mauri \cite{hitchingp}.

    We are interested in a class of fibrations constructed from elliptic surfaces.
    Let $\pi:X\to C$ be an elliptic surface.
    We will assume that $X,C$ are both smooth and projective.
    That said, our main theorems are all local over $C$.

    There is a map $X^{(n)}\to C^{(n)}$ between the symmetric powers of $X$ and $C$.
    Composing it with the Hilbert-Chow morphism gives a fibration
    \[\pi^{[n]}:X^{[n]}\to C^{(n)}=C^{[n]}.\]

    We will consider the case where $\pi$ has integral fibers and a section $s$.
    Then $\pi^{[n]}$ has a section given by $s^{[n]}:C^{[n]}\to X^{[n]}$.

    A generic fiber of $\pi^{[n]}$ would simply be a product of elliptic curves.
    Special fibers of this fibration are, however, rather ill-behaved:
    Even when $\pi$ is smooth, they can be neither reduced nor irreducible.

    Nevertheless, we establish the entire package of properties in this case, thus producing a class of examples with non-reduced fibers where this is done.

    The first is the Liouville theorem.
    \begin{theorem}
        There is an open $X^{[n],\pi}\subset X^{[n]}$ and a map $\mu:X^{[n],\pi}\times_{C^{[n]}} X^{[n]}\to X^{[n]}$ that restricts to a commutative group law on $X^{[n],\pi}$ with identity $s^{[n]}$, such that $\mu$ becomes a group scheme action extending the regular action.
    \end{theorem}
    In addition to its existence, the action can be interpreted geometrically.
    Recall that there is a Fourier-Mukai autoequivalence of the elliptic surface relative to its base, identifying it with its own relative compactified Jacobian.
    The transform of a length $n$ subscheme $Z$ is a sheaf supported on a possibly non-reduced curve, namely the preimage under $\pi$ of the length $n$ subscheme on $C$ given by $\pi^{[n]}(Z)$.
    The group scheme $X^{[n],\pi}$ turns out to be the locus where this Fourier-Mukai image is an invertible sheaf on said curve.
    And, on $k$-points, the group action can be described as taking tensor products there.

    This construction is done in Section \ref{ss:groupschemepart} along with a concrete description of finite subschemes belonging to $X^{[n],\pi}$.

    It is not obvious how the action extends to families.
    We address this issue in Section \ref{ss:algebraicity}.
    Additionally, we make the observation that the group action is a $\delta$-regular abelian fibration (Remark \ref{deltareg}), which may be of independent interest.

    The second part of the package is autoduality.
    \begin{theorem}
        There is a maximal Cohen-Macaulay sheaf $\mathcal P_n^{\rm BKR}$ on $X^{[n]}\times_{C^{[n]}}X^{[n]}$, invariant under swapping the factors, that induces an exact autoequivalence of $D^b\operatorname{Coh}(X^{[n]})$.
    \end{theorem}
    The construction follows an idea of Groechenig \cite{groec}, who proved the corresponding result for certain parabolic Hitchin systems.

    The $n$-fold outer tensor of the Fourier-Mukai kernel on the elliptic surface gives rise to an autoequivalence $\mathcal F^{\boxtimes n}$ of $D^b\operatorname{Coh}_{S_n}(X^n)$ by the work of Ploog \cite{ploog}.

    On the other hand, the derived McKay correspondence developed by Bridgeland, King, and Reid \cite{BKR} shows that the structure sheaf of the isospectral Hilbert scheme $\tilde{X}_n\hookrightarrow X^{[n]}\times X^n$ induces an equivalence
    \[\operatorname{BKR}:D^b\operatorname{Coh}(X^{[n]})\to D^b\operatorname{Coh}([X^n/S_n])=D^b\operatorname{Coh}_{S_n}(X^n).\]
    The same kernel induces an equivalence $\operatorname{BKR}^\top:D^b\operatorname{Coh}_{S_n}(X^n)\to D^b\operatorname{Coh}(X^{[n]})$ the other way around, as studied by Krug \cite{krug}.

    The sheaf $\mathcal P_n^{\rm BKR}$ is the kernel of the composition $\operatorname{BKR}^\top\circ\mathcal F^{\boxtimes n}\circ\operatorname{BKR}$.
    The claim that it has the properties we need is shown in Section \ref{ss:BKR}.

    The choice of taking a $\operatorname{BKR}$-congruence instead of a conjugate is consistent with that of \cite{groec,markman24}, and allows the kernel to have the desired properties.
    The difference between the two can be measured using \cite[Eq.~(11.5)]{markman24}.

    The last piece, which is what ties these together, is the theorem of the square.
    \begin{theorem}[Theorem of the square]
        The sheaf $\mathcal P_n^{\rm BKR}$ restricts to a line bundle $\mathcal P_n^{\rm BKR,\circ}$ on ${X^{[n],\pi}\times_{C^{[n]}}X^{[n]}}$, and the identity
        \[(\mu\times\operatorname{id})^\ast\mathcal P^{\rm BKR}_n=\operatorname{pr}_{13}^\ast\mathcal P^{\rm BKR,\circ}_n\otimes\operatorname{pr}_{23}^\ast\mathcal P^{\rm BKR}_n\]
        holds on $X^{[n],\pi}\times_{C^{[n]}}X^{[n]}\times_{C^{[n]}}X^{[n]}$.
    \end{theorem}
    This is proved in Section \ref{ss:tots}.
    Since $\mathcal P_n^{\rm BKR}$, and hence both sides of the equation, are maximal Cohen-Macaulay, this identity can be checked away from a closed subset of codimension $2$.
    Contrary to the case of compactified Jacobian fibrations, the complement of the smooth locus of this relative product has codimension $1$.
    So a more complicated open was used to prove this.

    Finally, we investigate in Section \ref{ss:tatesha} how this package helps us understand the Hilbert scheme of elliptic surfaces without a section (but still with integral fibers).
    The action $\mu$ allows sections of $X^{[n],\pi}$ to act by isomorphisms on $X^{[n]}$ relative to the base.
    So any $1$-cocycle $\Sigma$ of analytic local sections of $X^{[n],\pi}$ defines a Tate-Shafarevich twist $(X^{[n]})^\Sigma$ of $X^{[n]}$ relative to $C^{[n]}$.
    \begin{theorem}
        Suppose $Y\to C$ is an elliptic surface with integral fibers and $X\to C$ is its compactified Jacobian.
        Then:
        \begin{enumerate}
            \item There is some $1$-cocycle $\Sigma$ of sections of $X^{[n],\pi}$ such that $Y^{[n]}=(X^{[n]})^\Sigma$.
            \item Under this Tate-Shafarevich twist, the sheaf $\mathcal P_n^{\operatorname{BKR}}$ glues to a $(1\boxtimes\alpha)$-twisted sheaf on $Y^{[n]}\times_{C^{[n]}}X^{[n]}$ for some Brauer class $\alpha$ on $X^{[n]}$, and it induces an exact equivalence
            \[D^b\operatorname{Coh}(Y^{[n]})\cong D^b\operatorname{Coh}(X^{[n]},\alpha).\]
        \end{enumerate}
    \end{theorem}
    The derived equivalence itself may also be obtained using the techniques of \cite{bothuy}.
    \subsubsection*{Acknowledgements}
    The author would like to thank Junliang Shen for his encouragement and valuable suggestions, and David Fang, Weite Pi, Soumik Ghosh, Matthew Huynh for helpful discussions.
    The author would also like to thank Andres Fernandez Herrero for pointing out an error in a previous version of this paper.

    The author is grateful to the referees for their careful reading and for many helpful suggestions.

    The author was partially supported by NSF grant DMS-2301474.
    \section{The group scheme part of the Hilbert scheme}\label{ss:groupschemepart}
    Let $\pi:X\to C$ be an elliptic surface with integral fibers and a section $s:C\to X$.
    We assume that $X,C$ are both smooth projective.

    Singular fibers of $X$ are then either nodal or cuspidal cubics, with a unique singular point not on the section.
    Removing these points gives a smooth fibration $X^\circ\to C$.

    In \cite{bridgeland1998}, Bridgeland constructed a sheaf $\mathcal P\in\operatorname{Coh}(X\times_CX)$ that induces an autoequivalence $\mathcal F$ of $D^b\operatorname{Coh}(X)$ extending the usual Poincar\'e sheaf on the smooth fibers.

    The transform $\mathcal F$ makes $X/C$ self-dual, in the following sense:
    \begin{theorem}[\cite{bridgeland1998}, \cite{cubicdual}]\label{surfacegroupscheme}
        For any $p\in C$, the assignment $x\mapsto\mathcal F(k(x))$ identifies $X_p$ with the space of torsion-free sheaves of rank $1$ and degree $0$ on $X_p$.

        Moreover, the open subvariety consisting of invertible sheaves corresponds precisely to $X^\circ_p$, with $\mathcal F(k(s(p)))=\mathcal O_{X_p}$.
        And tensor product makes $X^\circ\to C$ a group scheme with unit section $s$ whose action on itself extends to $X$.
    \end{theorem}

    The starting point of our construction is the consideration of the image of connected finite subschemes under $\mathcal F$.
    We will repeatedly and implicitly use the following elementary fact:
    \begin{lemma}
        Let $\mathcal A$ be an abelian category with enough injectives and $A,B\in\mathcal A$.
        Suppose $E\in D(\mathcal A)$ fits in an exact triangle of the form
        \[E\to A\to B[1]\to E[1],\]
        then $E$ is (quasi-isomorphic to) an object of $\mathcal A$, and it is an extension of $A$ by $B$.
    \end{lemma}
    \begin{proof}
        The isomorphism $\operatorname{Ext}^1(A,B)\to\operatorname{Hom}_{D(\mathcal A)}(A,B[1])$ sends an exact sequence
        \[0\to B\to E\to A\to 0\]
        to the connecting homomorphism $A\to B[1]$.
    \end{proof}
    To start with, we describe each connected finite subscheme as a flag.
    \begin{proposition}
        A coherent sheaf $\mathcal G$ on $X$ is isomorphic to a structure sheaf of a length $n$ subscheme supported at a point $x$ if and only if there is a chain of surjections $\mathcal G=\mathcal G_n\to\cdots\to\mathcal G_0=0$ such that each kernel is isomorphic to $k(x)$, and the extension class of each of the exact sequences
        \[0\to k(x)\to \mathcal G_i\to\mathcal G_{i-1}\to 0\]
        is nonzero.
    \end{proposition}
    \begin{proof}
        If $Z=\operatorname{Spec}A$ is a length $n$ subscheme, then we can take a composition series of $A$ over itself.
        The only nonzero simple module over $A$ is $k$, so this gives a one-dimensional ideal of $A$.
        Repeating this process yields the chain.

        Conversely, it suffices to show that if $Y$ is a subscheme of length $n-1$ supported at $x$, then every nontrivial extension in $\operatorname{Ext}_X^1(\mathcal O_Y,k(x))$ is isomorphic to the structure sheaf of a length $n$ subscheme.
        But $\operatorname{Ext}_Y^1(\mathcal O_Y,k(x))=0$, so every nontrivial extension $M$ is scheme-theoretically supported on some length $n$ subscheme $Z$.
        Denote by $x\in M$ a lift of $1\in \mathcal O_Y$ under the projection $M\to\mathcal O_Y$.
        Then the composition
        \[\begin{tikzcd}
            \mathcal O_Z\arrow{r}{f(a)=ax} &[2em] M\rar & \mathcal O_Y
        \end{tikzcd}\]
        is just the projection $\mathcal O_Z\to\mathcal O_Y$.
        So the annihilator of $x$ is either $0$ or $\mathcal I_Y$.
        If it is the latter, then $f$ factors through $\mathcal O_Y\to M$ which splits the sequence.
        So it must be the former.
        $f$ is then an isomorphism $\mathcal O_Z\cong M$.
    \end{proof}
    \begin{corollary}\label{ext1_checkhilb}
        Let $\mathcal G$ be a coherent sheaf on $X$.
        The following are equivalent:
        \begin{enumerate}
            \item $\mathcal G\cong\mathcal F(\mathcal O_Z)$ for some length $n$ subscheme $Z$.
            \item There exists a partition $n=n_1+\cdots+n_r$, distinct points $x_1,\ldots,x_r$, and a decomposition $\mathcal G=\mathcal G^1\oplus\cdots\oplus\mathcal G^r$ such that there is a chain of surjections $\mathcal G^j=\mathcal G^j_{n_j}\to\cdots\to\mathcal G^j_0=0$ each with kernel isomorphic to $\mathcal F(k(x_j))$, and the extension class of the exact sequence
            \[0\to \mathcal F(k(x_j))\to \mathcal G_i^j\to\mathcal G_{i-1}^j\to 0\]
            is nonzero.
        \end{enumerate}

        Moreover, we can decompose $Z=Z_1\sqcup\cdots\sqcup Z_r$ into connected components such that $Z_j$ is supported at $x_j$, has length $n_j$, and satisfies $\mathcal F(\mathcal O_{Z_j})=\mathcal G^j$.
    \end{corollary}
    For $p\in C$, we write $np$ for the unique length $n$ subscheme of $C$ supported at $p$, and $nX_p$ for the preimage scheme of $np$.
    Equivalently, $nX_p$ is the effective Cartier divisor of multiplicity $n$ supported on $X_p=1X_p$.

    By Corollary \ref{ext1_checkhilb}, we should expect certain invertible sheaves on $nX_p$ to be Fourier-Mukai images of connected finite subschemes.
    One instance of this is $\mathcal O_{nX_p}$.
    \begin{lemma}\label{relative_fm_bc}
        Suppose $\pi:X\to C, \check\pi:\check{X}\to C$ are proper and flat, where $X,\check{X},C$ are Deligne-Mumford stacks.
        Let $\mathcal P\in D\operatorname{QCoh}(\check{X}\times_C X)$ and $\mathcal F=\Phi_{\mathcal P}$.

        (a) For any $\epsilon\in \operatorname{Perf}(C)$, we have
        \[
            \mathcal F(-\otimes_{\mathcal O_{\check{X}}}^{\mathbf L}\check\pi^\ast\epsilon)=\mathcal F(-)\otimes_{\mathcal O_X}^{\mathbf L}\pi^\ast\epsilon.
        \]

        (b) For any closed immersion $W\hookrightarrow C$, write $i:X_W\hookrightarrow X,\check{i}:\check{X}_W\hookrightarrow\check{X}$.
        Then $\mathcal F\circ\check{i}_\ast=i_\ast\circ\mathcal F_W$, where $\mathcal F_W$ is the Fourier-Mukai transform with kernel $\mathbf L(\check{i}\times i)^\ast\mathcal P$, and $\check{i}\times i:\check X_W\times_WX_W\to \check{X}\times_CX$.
    \end{lemma}
    We refer the reader to \cite[Corollary 4.12]{hallrydh} for a general version of the projection formula, of which the lemma is an application.
    Note that the lemma takes place in the category $D\operatorname{QCoh}$, but for all of our applications the objects will stay in $D^b\operatorname{Coh}$.
    \begin{proof}
        Fix the notations of projections and closed embeddings as follows:
        \[
        \begin{tikzcd}
            \check{X}_W\arrow{d}{\check{i}}& \check{X}_W\times_WX_W\arrow{r}{p_W}\arrow[swap]{l}{\check{p}_W}\arrow{d}{\check{i}\times i}&X_W\arrow{d}{i}\\
            \check{X}& \check{X}\times_CX\arrow{r}{p}\arrow[swap]{l}{\check{p}}&X
        \end{tikzcd}.
        \]
        Note that these squares are Cartesian.
        We then have
        \begin{align*}
            \mathcal F(A\otimes^{\mathbf L}\check\pi^\ast\epsilon)&=\mathbf Rp_\ast(\mathcal P\otimes^{\mathbf L}\check{p}^\ast(A\otimes^{\mathbf L}\check\pi^\ast\epsilon))=\mathbf Rp_\ast(\mathcal P\otimes^{\mathbf L}\check{p}^\ast A\otimes^{\mathbf L}\check{p}^\ast\check\pi^\ast\epsilon)\\
            &=\mathbf Rp_\ast(\mathcal P\otimes^{\mathbf L}\check{p}^\ast A\otimes^{\mathbf L}p^\ast\pi^\ast\epsilon)=\mathbf Rp_\ast(\mathcal P\otimes^{\mathbf L}\check{p}^\ast A)\otimes^{\mathbf L}\pi^\ast\epsilon\\
            &=\mathcal F(A)\otimes^{\mathbf L}\pi^\ast\epsilon.
        \end{align*}
        \begin{align*}
            \mathcal F(\check{i}_\ast A)&=\mathbf Rp_\ast(\mathcal P\otimes^{\mathbf L}\check{p}^\ast\check{i}_\ast A)=\mathbf Rp_\ast(\mathcal P\otimes^{\mathbf L}(\check{i}\times i)_\ast\check{p}_W^\ast A)\\
            &=\mathbf Rp_\ast(\check{i}\times i)_\ast(\mathbf L(\check{i}\times i)^\ast\mathcal P\otimes^{\mathbf L}\check{p}_W^\ast A)=i_\ast\mathbf R(p_W)_\ast(\mathbf L(\check{i}\times i)^\ast\mathcal P\otimes^{\mathbf L}\check{p}_W^\ast A).\qedhere
        \end{align*}
    \end{proof}
    \begin{proposition}\label{section_of_fibration}
        We have $\mathcal F(\mathcal O_{s(np)})=\mathcal O_{nX_p}$.
    \end{proposition}
    \begin{proof}
        Observe that
        \[s_\ast\mathcal O\otimes^{\mathbf L}\pi^\ast\mathcal O_{np}=s_\ast(\mathbf Ls^\ast\pi^\ast\mathcal O_{np})=s_\ast\mathcal O_{np}=\mathcal O_{s(np)}.\]
        We also know that $\mathcal F(s_\ast\mathcal O)$ is an invertible sheaf pulled back from $C$, say $\mathcal F(s_\ast\mathcal O)=\pi^\ast \Theta$.
        So
        \[\mathcal F(s_\ast\mathcal O)\otimes^{\mathbf L}\pi^\ast\mathcal O_{np}=\pi^\ast(\Theta\otimes\mathcal O_{np})=\pi^\ast\mathcal O_{np}=\mathcal O_{nX_p}.\]
        The claim then follows from the preceding lemma.
    \end{proof}
    \begin{proposition}\label{multiply}
        Suppose $L$ is an invertible sheaf on $nX_p$ such that $L_{\rm red}:=L|_{X_p}$ has degree $0$.
        Let $Z$ be any connected finite subscheme of length $n$ with $\pi(Z_{\rm red})=p$.
        Write $\mathcal F(\mathcal O_Z)=(i_{nX_p})_\ast N$.

        Then $(i_{nX_p})_\ast L\otimes^{\rm und} \mathcal F(\mathcal O_Z)=(i_{nX_p})_\ast(L\otimes N)$ is isomorphic to $\mathcal F(\mathcal O_{Z'})$ for some connected finite subscheme $Z'$ of length $n$.
        And $\mathcal F(k(Z_{\rm red}'))=(i_{X_p})_\ast(L_{\rm red}\otimes N_{\rm red})$ where $N_{\rm red}:=N|_{X_p}$.
    \end{proposition}
    Here, and later in this article, the notation $\otimes^{\rm und}$ emphasizes that we are using the non-derived tensor product.
    \begin{proof}
        By Corollary \ref{ext1_checkhilb}, there is a chain of surjections $(i_{nX_p})_\ast N=\mathcal G_n\to\cdots\to\mathcal G_0=0$ with kernels all isomorphic to $(i_{X_p})_\ast N_{\rm red}$.
        Each $\mathcal G_i$ is then supported on $nX_p$, so we can write $\mathcal G_i=(i_{nX_p})_\ast\mathcal H_i$ for some coherent sheaves $\mathcal H_i$ on $nX_p$.

        Then take $\mathcal G_i^L=(i_{nX_p})_\ast(\mathcal H_i\otimes L)$.
        We then have a chain of surjections $(i_{nX_p})_\ast (L\otimes N)=\mathcal G_n^L\to\cdots\to\mathcal G_0^L=0$ with kernels
        \[(i_{nX_p})_\ast ((i_{X_p})_\ast N_{\rm red}\otimes L)=(i_{nX_p})_\ast (i_{X_p})_\ast (N_{\rm red}\otimes L_{\rm red})=(i_{X_p})_\ast (N_{\rm red}\otimes L_{\rm red}).\qedhere\]
    \end{proof}
    Combined with Proposition \ref{section_of_fibration}, we obtain:
    \begin{corollary}
        Suppose $L$ is an invertible sheaf on $nX_p$ such that $L_{\rm red}=L|_{X_p}$ has degree $0$.
        Then $(i_{nX_p})_\ast L$ is isomorphic to $\mathcal F(\mathcal O_Z)$ for some connected finite subscheme $Z$ of length $n$, and $\mathcal F(k(Z_{\rm red}))=(i_{X_p})_\ast L_{\rm red}$.
    \end{corollary}
    The converse is not true:
    Not all connected finite subscheme $Z$ of length $n$ has Fourier-Mukai image equal to an invertible sheaf on $nX_{\pi(Z_{\rm red})}$, even when $Z\hookrightarrow X^\circ$.
    But it turns out that we do have a nice description of those that do.
    \begin{proposition}\label{geom_desc}
        Let $Z$ be a connected finite subscheme of length $n$ supported on $x\in X$.
        Write $p=\pi(x)$.
        Let $W\hookrightarrow np\hookrightarrow C$ be the scheme-theoretic image of $\pi|_Z$.
        The following are equivalent:
        \begin{enumerate}
            \item The subscheme $Z$ belongs to $X^\circ$ and $W=np$.
            \item The transform $\mathcal F(\mathcal O_Z)$ is an invertible sheaf on $nX_p$.
        \end{enumerate}
        Furthermore, if this happens, then $(i_{X_p})_\ast(\mathcal F(\mathcal O_Z)|_{X_p})=\mathcal F(k(x))$.
    \end{proposition}
    \begin{proof}
        If $x\in X\setminus X^\circ$, then $x$ is the singularity on $X_p$.
        By Corollary \ref{ext1_checkhilb}, $\mathcal F(\mathcal O_Z)$ surjects onto $\mathcal F(k(x))$.
        So $i_x^\ast \mathcal F(\mathcal O_Z)$ surjects onto $i_x^\ast\mathcal F(k(x))$ which has dimension at least $2$, so $\mathcal F(\mathcal O_Z)$ is not an invertible sheaf on its support.

        If $W=mp$ with $m<n$, then $\mathcal O_Z$ is supported on $mX_p$, so $\mathcal F(\mathcal O_Z)$ is supported on $mX_p$ by Lemma \ref{relative_fm_bc}.
        So it cannot be an invertible sheaf on $nX_p$.

        Conversely, suppose $x\in X^\circ$ and $W=np$.
        In this case $\pi|_Z:Z\to W$ has to be an isomorphism.
        Indeed, if the associated map $\mathcal O_W\to\mathcal O_Z$ is not an isomorphism, then it cannot be injective as both are finite-dimensional over $k$, which means that $\pi|_Z$ factors through a smaller subscheme of $W$, a contradiction.

        Hence the second projection $r:Z\times_WX_W\to X_W$ is an isomorphism.
        Then,
        \[
        \mathbf R(\operatorname{pr}_2)_\ast(\mathcal P\otimes^{\mathbf L}\operatorname{pr}_1^\ast\mathcal O_Z)=\mathbf R(\operatorname{pr}_2\circ i_{Z\times_WX_W})_\ast\mathbf Li_{Z\times_WX_W}^\ast\mathcal P=(i_{X_W})_\ast r_\ast\mathbf Li_{Z\times_WX_W}^\ast\mathcal P.
        \]

        Now, since $x\in X^\circ$, $Z$ is a closed subscheme of $X^\circ$, so $Z\times_WX_W$ is a closed subscheme of $X^\circ\times_CX$.
        On the other hand, $\mathcal P|_{X^\circ\times_C X}$ is an invertible sheaf.
        So $\mathbf Li_{Z\times_WX_W}^\ast\mathcal P=i_{Z\times_WX_W}^\ast\mathcal P$ is an invertible sheaf.

        We moreover have
        \[\mathcal F(\mathcal O_Z)|_{X_p}= i_{X_p}^\ast(i_{X_W})_\ast r_\ast i_{Z\times_WX_W}^\ast\mathcal P= i_{X_p}^\ast (r^{-1})^\ast i_{Z\times_WX_W}^\ast\mathcal P=(r_x^{-1})^\ast i_{\{x\}\times X_p}^\ast\mathcal P,
        \]
        where $r_x:\{x\}\times X_p\to X_p$.
        Therefore $(i_{X_p})_\ast(\mathcal F(\mathcal O_Z)|_{X_p})=\mathcal F(k(x))$.
    \end{proof}
    \begin{definition}
        The group scheme part of $X^{[n]}$ relative to $\pi$, denoted $X^{[n],\pi}\subset X^{[n]}$, is defined to be the open subvariety consisting of length $n$ subschemes $Z\hookrightarrow X^\circ$ such that the scheme-theoretic image of $\pi|_Z$ has length $n$.
    \end{definition}
    Suppose we decompose $Z=Z_1\sqcup\cdots\sqcup Z_r$ into connected components.
    Then $Z\in X^{[n],\pi}$ iff $\pi((Z_i)_{\rm red})$ are pairwise distinct, and the equivalence conditions in the preceding proposition apply to each $Z_i$.

    Write $\pi^{[n]}:X^{[n]}\to C^{[n]}$ for the composition of $X^{(n)}\to C^{(n)}=C^{[n]}$ with the Hilbert-Chow morphism $X^{[n]}\to X^{(n)}$.
    \begin{corollary}
        For $W=n_1p_1+\cdots+n_rp_r\in C^{[n]}$, we write $E_W$ for the effective Cartier divisor $n_1X_{p_1}+\cdots +n_rX_{p_r}$.
        Then $Z\in X^{[n],\pi}_W$ if and only if $\mathcal F(\mathcal O_Z)$ is (the pushforward of) an invertible sheaf on $E_W$ with degree $0$ on the reduction of each component.
    \end{corollary}

    We can then define the action map on $k$-points
    \[\mu_k:(X^{[n],\pi}\times_{C^{[n]}}X^{[n]})(k)\to X^{[n]}(k),(\mathcal O_{Z_1},\mathcal O_{Z_2})\mapsto \mathcal F^{-1}(\mathcal F(\mathcal O_{Z_1})\otimes_{\mathcal O_X}^{\rm und}\mathcal F(\mathcal O_{Z_2})). \]

    Equivalently, suppose $Z_1,Z_2\in X^{[n],\pi}_W$, then $\mathcal F(\mathcal O_{Z_i})=(i_{E_W})_\ast L_i$ for $L_i=i_{E_W}^\ast \mathcal F(\mathcal O_{Z_i})\in\operatorname{Coh}(E_W)$ (Lemma \ref{relative_fm_bc}), and the action map takes $(\mathcal O_{Z_1},\mathcal O_{Z_2})$ to $\mathcal F^{-1}((i_{E_W})_\ast( L_1\otimes L_2))$.
    Hence the map is well-defined by Proposition \ref{multiply}.

    From its definition, it is unclear that we can extend $\mu$ to families, i.e.~that it is algebraic.
    We will address this in the next section.

    The involution on $X$ relative to $C$ that dualizes a torsion-free sheaf of rank $1$ on a fiber gives rise to an involution $\nu$ of $X^{[n]}$ relative to $C^{[n]}$.
    Clearly, $\nu$ restricts to an involution $\nu|_{X^{[n],\pi}}$ of $X^{[n],\pi}$.
    \begin{proposition}
        Assuming the map $\mu$ is algebraic, then:
        \begin{enumerate}
            \item $\mu|_{X^{[n],\pi}\times_{C^{[n]}} X^{[n],\pi}}$ makes $X^{[n],\pi}$ a commutative group scheme over $C^{[n]}$, with inversion $\nu|_{X^{[n],\pi}}$ and identity section $s^{[n]}:C^{[n]}\to X^{[n]}$, and
            \item $X^{[n],\pi}$ acts, as a group scheme, on $X^{[n]}$ via $\mu$.
        \end{enumerate}
    \end{proposition}
    \begin{proof}
        The axioms are identities involving morphisms between reduced varieties, so it is enough to verify them on the $k$-points of a Zariski-dense open.
        In particular, one may check them after a base-change to the open locus in $C^{[n]}$ consisting of reduced subschemes $W$ such that $X_W$ is smooth.
        Both claims are immediate there.
    \end{proof}
    \section{Algebraicity of the group action}\label{ss:algebraicity}
    An obvious way to attempt the extension of the definition of $\mu$ to families is as follows:
    For any scheme $T$ (which we may assume to be a variety, as $X^{[n],\pi}\times_{C^{[n]}}X^{[n]}$, $X^{[n]}$, and $X^{[n],\pi}$ all are), one might attempt to define $\mu_T$ with the same formula except that we apply a base-changed version $\mathcal F_T$ of our original Fourier-Mukai transform instead.
    \begin{definition}
        For any $T$ as above, we write $\mathcal F_T$ for the transform with kernel $\mathcal P_T\in\operatorname{Coh}((X\times_CX)\times T) =\operatorname{Coh}((X\times T)\times_{C\times T}(X\times T))$
    \end{definition}

    There are some difficulties to make this work directly.
    In general, a sheaf on $X\times T$ flat over $T$ is not determined, up to the action of $\operatorname{Pic}(T)$, by its values on the fibers under the second projection.
    \begin{example}
        Take two elliptic curves $E,E'$ and consider the Poincar\'e sheaf relative to $E\times E'\to E'$.
        Choose any length $2$ subscheme concentrated at $(e,e')$ that is not contained in $\{e\}\times E'$ nor $E\times\{e'\}$ and consider its Fourier-Mukai image $\mathcal G$.
        By Proposition \ref{geom_desc}, it is the pushforward of an invertible sheaf on $2(E\times \{e'\})$ whose restriction to $E\times \{e'\}$ is isomorphic to $\mathcal O_{E\times \{e'\}}$.

        Now this sheaf and $\mathcal O_{2(E\times \{e'\})}$ are both flat over $E$, and their restrictions to the fibers are both isomorphic to $\mathcal O_{2e'}$.
        But they are not in the same $\operatorname{Pic}(E)$-orbit as they both restrict to the trivial line bundle on $E\times \{e'\}$.
    \end{example}
    To resolve this, we need to view Hilbert schemes via ideals instead.
    Consider the moduli space $\mathcal M$ whose values on $T$-points are coherent sheaves on $X\times T$ flat over $T$ whose fibers are rank $1$ torsion-free sheaves with trivial determinant, modulo action of $\operatorname{Pic}(T)$.

    The fibers will have constant Euler characteristic $\chi$ over $T$.
    This decomposes $\mathcal M$ into connected components $M_{\chi(\mathcal O)-\chi}$.
    As points have codimension $2$ in $X$, we have a well-known isomorphism $X^{[n]}\to\mathcal M_n$ (see e.g.~\cite[Example 4.3.6]{Huybrechts_Lehn}) by forgetting the inclusion map of the ideal sheaf into $\mathcal O_X$, i.e.
    \[
    X^{[n]}(T)\to \mathcal M_n(T),[Z\hookrightarrow X\times T]\mapsto [\mathcal I_Z].
    \]

    Let us first understand the pointwise image of these ideal sheaves.
    Write $\mathcal F(\mathcal O)=s_\ast\Theta[-1]$ where $\Theta$ is a line bundle on $C$.
    For a finite subscheme $Z\subset X$ of length $n$, we apply $\mathcal F$ to the corresponding exact triangle.
    This gives an exact triangle
    \[\mathcal F(\mathcal I_Z)\to s_\ast\Theta[-1]\to\mathcal F(\mathcal O_Z)\to\mathcal F(\mathcal I_Z)[1].\]
    So $\mathcal F(\mathcal I_Z)[1]$ must be a sheaf that fits into an exact sequence
    \[0\to \mathcal F(\mathcal O_Z)\to \mathcal F(\mathcal I_Z)[1]\to s_\ast\Theta\to 0.\]
    Note that this means, in particular, that if $Z\in X^{[n]}_{n_1p_1+\cdots+n_rp_r}$, then $\mathcal F(\mathcal I_Z)$ is supported on the effective Cartier divisor $D_{n_1p_1+\cdots+n_rp_r}=s+E_{n_1p_1+\cdots+n_rp_r}$.

    Conversely, we can completely characterize extensions of this form.
    \begin{proposition}\label{ext2_checkhilb}
        Suppose we have an extension
        \[0\to \mathcal F(\mathcal O_Z)\to \mathcal G\to s_\ast\Theta\to 0.\]
        Then $\mathcal G\cong\mathcal F(\mathcal I_Z)[1]$ if and only if $\mathcal G$ is simple, i.e.~$\dim\operatorname{End}(\mathcal G)=1$.
    \end{proposition}
    \begin{proof}
        Let $\mathcal C=\mathcal F^{-1}(\mathcal G[-1])$.
        Then we have an exact triangle $\mathcal C\to \mathcal O\to\mathcal O_Z\to\mathcal C[1]$.

        If $\mathcal O\to \mathcal O_Z$ is surjective, then its kernel is forced to be $\mathcal I_Z$, so $\mathcal C\cong\mathcal I_Z$.
        Note that in this case $\dim\operatorname{End}(\mathcal G)=\dim\operatorname{End}(\mathcal I_Z)=1$ as $\mathcal I_Z$ is torsion-free of rank $1$.

        Otherwise, the image of this map, as a quotient of $\mathcal O$ and a subsheaf of $\mathcal O_Z$, must be isomorphic to $\mathcal O_{Z'}$ for some finite length subscheme $Z'$; the cokernel, on the other hand, is a quotient of $\mathcal O_Z$, and therefore too is isomorphic to $\mathcal O_{Z''}$ for some other finite length (nonempty) subscheme $Z''$.

        Suppose $Z$ is reduced, then $\mathcal C\cong\mathcal I_{Z'}\oplus\mathcal O_{Z''}[-1]$, which has nontrivial endomorphisms since it is a direct sum of nonzero objects.
        Hence $\mathcal G=\mathcal F(\mathcal C)[1]$ cannot be simple either.

        In general, let $\mathcal Z\hookrightarrow X\times X^{[n]}$ be the universal subscheme, then $(\operatorname{pr}_{X^{[n]}})_\ast\mathcal O_{\mathcal Z}$ is a vector bundle of rank $n$.
        Choose a small open $U$ around $Z$ that trivializes it.
        The object $\mathcal C\cong[\mathcal O\to\mathcal O_Z]$ is then the specialization of an object of the form $[\mathcal O_{X\times U}\to \mathcal O_{\mathcal Z}|_U]$.
        The claim then follows from upper semi-continuity of $\dim\operatorname{End}(\mathcal C)=\dim\mathbb H^0\mathbf R\underline{\operatorname{Hom}}(\mathcal C,\mathcal C)$ (cf.~\cite[Proposition 6.4]{semicont}).
    \end{proof}
    \begin{remark}
        Given any $\mathcal G\cong\mathcal F(\mathcal I_Z)[1]$, note that $\operatorname{Hom}(\mathcal G,s_\ast\Theta)=\operatorname{Hom}(\mathcal I_Z,\mathcal O)$ has dimension $1$, so we can always recover $\mathcal F(\mathcal O_Z)$ this way.
    \end{remark}
    Let us now identify points in $X^{[n],\pi}$ using the images of their ideals.
    \begin{definition}
        We call an invertible sheaf over $D_{n_1p_1+\cdots+n_rp_r}$ strongly numerically trivial if it restricts to $\mathcal O$ on $s(C)$, and to a degree $0$ invertible sheaf on $X_{p_i}$ for each $i$.
    \end{definition}
    In particular, a strongly numerically trivial invertible sheaf $L$ over $D_W$ would have $(i_{E_W})_\ast L|_{E_W}=\mathcal F(\mathcal O_Z)$ for some $Z\in X^{[n],\pi}_W$.

    Let us now fix $W\in C^{[n]}$.
    For notational convenience, we write $E$ for $E_W$, $D$ for $D_W$, and $F$ for $s(W)=E_W\times_Xs(C)$.
    \begin{proposition}\label{multiply_ideals}
        Suppose $Z\in X^{[n]}_W$ and $L$ is a strongly numerically trivial invertible sheaf on $D$.
        Then $(i_D)_\ast L\otimes^{\rm und}\mathcal F(\mathcal I_Z)[1]=\mathcal F(\mathcal I_{Z'})[1]$ for some $Z'\in X^{[n]}_W$.
        Moreover, $\mathcal F(\mathcal O_{Z'})=(i_E)_\ast L|_E\otimes^{\rm und}\mathcal F(\mathcal O_Z)$.
    \end{proposition}
    \begin{proof}
        This is essentially Proposition \ref{multiply}, with the additional observation that multiplication by a line bundle does not change the dimension of the endomorphism ring.
    \end{proof}
    \begin{lemma}\label{pushoutDEF}
        The pushout $E\sqcup_Fs(C)$ is isomorphic to $D$.
    \end{lemma}
    \begin{proof}
        The natural morphism $E\sqcup_Fs(C)\to D$ is a monomorphism and is universally closed by the valuative criterion, so it must be a closed immersion.
        But it is also an isomorphism generically on each irreducible component and $D$ is pure, so it is an isomorphism.
    \end{proof}
    \begin{lemma}
        We have $\operatorname{Pic}(D)\cong \operatorname{Pic}(E)\times\operatorname{Pic}(s(C))$, with the projection maps given by pullbacks to $E$ and $s(C)$, respectively.
    \end{lemma}
    \begin{proof}
        We have $\mathcal O_D=\mathcal O_E\times_{\mathcal O_F}\mathcal O_{s(C)}$ by the preceding lemma.
        This induces an exact sequence
        \[
            1\to \mathcal O_D^\times\to (i_E)_\ast\mathcal O_E^\times\times (i_{s(C)})_\ast\mathcal O_{s(C)}^\times\to(i_F)_\ast\mathcal O_F^\times\to 1,
        \]
        where the second-to-last arrow is given by $(a,b)\mapsto ab^{-1}$.

        Note that $\operatorname{Pic}(F)=1$ since it is finite, so we have an exact sequence
        \[
        \Gamma(\mathcal O_E^\times)\times \Gamma(\mathcal O_{s(C)}^\times)\to \Gamma(\mathcal O_F^\times)\to\operatorname{Pic}(D)\to\operatorname{Pic}(E)\times\operatorname{Pic}(s(C))\to 1.
        \]
        Now the first map is surjective.
        Indeed, the map $\Gamma(\mathcal O_E^\times)\to \Gamma(\mathcal O_F^\times)$ given by restriction (i.e. pulling back via $s$) is surjective since it has a section given by pulling back via $\pi$.
        And this restriction map factors as $\Gamma(\mathcal O_E^\times)\to\Gamma(\mathcal O_E^\times)\times \Gamma(\mathcal O_{s(C)}^\times)\to \Gamma(\mathcal O_F^\times)$ with the first map given by $a\mapsto (a,1)$.
    \end{proof}
    \begin{proposition}\label{tensor_adjust}
        There is a bijection between strongly numerically trivial invertible sheaves $L$ on $D$ and sheaves of the form $\mathcal F(\mathcal I_Z)[1]$ for $Z\in X^{[n],\pi}_W$, given by the formula
        \[L\mapsto (i_{D})_\ast L\otimes(\pi^\ast\Theta\otimes \pi^\ast s^{\ast}\mathcal O(-s)\otimes\mathcal O(s)).\]
        Moreover, if $L$ maps to $\mathcal F(\mathcal I_Z)[1]$, then $(i_E)_\ast L|_E=\mathcal F(\mathcal O_Z)$.
    \end{proposition}
    \begin{proof}
        Let us first show that this is well-defined and that the last formula is true.
        In view of Proposition \ref{multiply_ideals}, it suffices to show that $(i_D)_\ast i_D^\ast(\pi^\ast\Theta\otimes \pi^\ast s^{\ast}\mathcal O(-s)\otimes\mathcal O(s))$ has the form given by Proposition \ref{ext2_checkhilb}.
        We start with the exact sequence
        \[
        0\to \mathcal O(-s)/\mathcal O(-D)\to \mathcal O_D\to s_\ast\mathcal O\to 0.
        \]
        Note that we have $\mathcal O(-s)/\mathcal O(-D)\cong\mathcal O(-s)|_E$.
        Twisting this sequence with $\pi^\ast\Theta\otimes \pi^\ast s^{\ast}\mathcal O(-s)\otimes\mathcal O(s)$ then gives
        \[
        0\to \mathcal O_E\to (i_D)_\ast i_D^\ast(\pi^\ast\Theta\otimes \pi^\ast s^{\ast}\mathcal O(-s)\otimes\mathcal O(s)) \to s_\ast\Theta\to 0.
        \]
        We need to show that the middle entry is simple.
        This is equivalent to saying that $\mathcal O_D$ is simple.
        
        By Lemma \ref{pushoutDEF}, $\operatorname{End}(\mathcal O_D)=\Gamma(\mathcal O_D)=\Gamma(\mathcal O_E\times_{\mathcal O_F}\mathcal O_{s(C)})=\Gamma(\mathcal O_E)\times_{\Gamma(\mathcal O_F)}\Gamma(\mathcal O_{s(C)})$.
        Now $\Gamma(\mathcal O_E)\to\Gamma(\mathcal O_F)$ has a section given by pulling back through $\pi$, so it is surjective.
        The latter has dimension $n$.
        On the other hand, $\mathcal O_E$ is the result of an $n$-fold extension by sheaves of the form $\mathcal O_{X_p}$, which has $1$-dimensional global sections.
        Hence $\dim\Gamma(\mathcal O_E)\le n$.
        Therefore $\Gamma(\mathcal O_E)\to\Gamma(\mathcal O_F)$ must be an isomorphism.
        This means that $\dim\operatorname{End}(\mathcal O_D)=\dim\Gamma(\mathcal O_C)=1$.

        Injectivity is clear.
        It remains to show surjectivity.
        For any $Z\in X_W^{[n],\pi}$, we know that $\mathcal F(\mathcal O_Z)\in\operatorname{Pic}(E_W)$ and has degree $0$ on the reduction of each component, so there is a strongly numerically trivial line bundle $L$ over $D_W$ such that $L|_{E_W}=\mathcal F(\mathcal O_Z)$ by the preceding lemma.
    \end{proof}
    We can then define a map $\mu_T:(X^{[n],\pi}\times_{C^{[n]}}X^{[n]})(T)\to X^{[n]}(T)$ for all varieties $T$, sending $([\mathcal I_{Z_1}],[\mathcal I_{Z_2}])$ to
    \[\mathcal F_T^{-1}(\mathcal F_T(\mathcal I_{Z_1})[1]\otimes^{\rm und}\mathcal F_T(\mathcal I_{Z_2})[1]\otimes \operatorname{pr}_X^\ast (\pi^\ast\Theta^\vee\otimes \pi^\ast s^{\ast}\mathcal O(s)\otimes\mathcal O(-s)))[-1]. \]
    This is now well-defined (flatness over $T$ is preserved due to Lemma \ref{relative_fm_bc}) and functorial.
    It also recovers the same operation as before when $T=\operatorname{Spec}k$.
    This establishes the algebraicity of $\mu$.
    \begin{remark}
        Switching to ideal sheaves allows one to convert pointwise information into statements over a base, thereby making the construction well-behaved in families.
        More generally, simple sheaves satisfy a seesaw principle similar to the classical version for line bundles.
        We will discuss this in more detail later at Proposition \ref{seesaw}.
    \end{remark}
    \begin{remark}\label{stablesh}
        One might wonder whether $\mathcal F(\mathcal I_Z)[1]$ can be viewed as (semi)stable one-dimensional sheaves of a certain kind.
        Indeed this is the case, for sufficiently large $n$ and a suitably chosen normalization of $\mathcal P$.
        This is one of the results of \cite[Theorem 3.15]{yos}.
        In particular, the group action may be interpreted as tensor product (adjusted appropriately as in Proposition \ref{tensor_adjust}) in a moduli space of (semi)stable sheaves on the curve family $D_W,W\in C^{[n]}$.
    \end{remark}
    \begin{remark}\label{deltareg}
        Let us analyze this group action on fibers.
        For any $p\in C$ and $n\ge 1$, we have an exact sequence
        \[
        1\to 1+I\to \mathcal O_{nX_p}^\times\to\mathcal O_{(n-1)X_p}^\times\to 1
        \]
        where $I$ is the (square-zero) ideal of $(n-1)X_p$ in $nX_p$.
        Note that $I$ is isomorphic to $\mathcal O_{X_p}$ as $\mathcal O_{nX_p}$-modules.
        So $H^0(1+I)=\mathbb G_a$ and $H^1(1+I)=0$, and therefore $\operatorname{Pic}(nX_p)$ is an extension of $\operatorname{Pic}((n-1)X_p)$ by $\mathbb G_a$.

        Inductively, if we take any $W\in C^{[n]}$, then the abelian part of $X^{[n],\pi}_W$ is the abelian part of $\operatorname{Pic}^0(E_W^{\rm red})$, which is a product of smooth fibers appearing in $E_W$.

        Consequently, the group action $\mu$ is a $\delta$-regular abelian fibration (see \cite{ngo} for the relevant definitions).
        Indeed, for any geometric point $Z'\in X^{[n]}_W$, if $Z\in X^{[n],\pi}_W$ stabilizes $Z'$, then necessarily the invertible sheaf on $E_W$ corresponding to $Z$ restricts to $\mathcal O$ on $E_W^{\rm red}$ by Proposition \ref{multiply}.
        Hence $Z$ is in the affine part of $X^{[n],\pi}_W$.
        The group scheme is polarizable by the general result of \cite{polaris}.
        Moreover, the affine part of $X^{[n],\pi}_W$ has dimension equal to $\delta=n-r_W$ where $r_W$ is the number of smooth fibers in $E_W$.
        So the set of $W$ such that the affine part of $X^{[n],\pi}_W$ is $\delta$-dimensional has dimension $n-\delta$ exactly.
        This establishes $\delta$-regularity.
    \end{remark}
    \section{The BKR equivalence and a Poincar\'e sheaf}\label{ss:BKR}
    \begin{definition}
        For an object $\mathcal Q\in D\operatorname{Coh}(M'\times M)$, its transpose is defined as $\mathcal Q^\top=\nu^\ast\mathcal Q$, where $\nu:M\times M'\to M'\times M$ is the morphism swapping the factors.
    \end{definition}
    \begin{example}
        The kernel $\mathcal P$ of the transform in Theorem \ref{surfacegroupscheme} satisfies $\mathcal P=\mathcal P^\top$.
    \end{example}
    The following lemma is immediate.
    \begin{lemma}\label{transpose}
        We have $(\mathcal Q\circ\mathcal R)^\top=\mathcal R^\top\circ\mathcal Q^\top$.
    \end{lemma}
    Let $X$ be a smooth variety of dimension at most $2$.
    The symmetric group $S_n$ acts on $X^n$ by permuting the coordinates.
    We consider the $S_n$-isospectral Hilbert scheme $\tilde{X}_n\hookrightarrow X^{[n]}\times X^n$ on $X^n$.
    It is normal and Gorenstein by \cite[Theorem 1, \S3.8]{haiman}.
    And it comes equipped with morphisms
    \[
    \begin{tikzcd}
        X^{[n]}& {[}\tilde{X}_n/S_n{]} \arrow{l}[swap]{q}\arrow{r}{p}&{[}X^n/S_n{]}
    \end{tikzcd},
    \]
    such that $p$ is birational and $q$ is a finite flat $S_n$-quotient.
    \begin{remark}\label{BKRcurves}
        When $X$ has dimension $1$, i.e.~is a smooth curve, we would simply have $\tilde{X}_n\cong X^n$ via $p$, and $q$ is the quotient $[X^n/S_n]\to X^{(n)}$ with $S_n$ acting on the coordinates.
    \end{remark}

    Bridgeland-King-Reid \cite{BKR} showed that the transform with kernel $\mathcal O_{\tilde{X}_n}$, namely
    \[\operatorname{BKR}=\mathbf Rp_\ast\circ q^\ast:D^b\operatorname{Coh}(X^{[n]})\to D^b\operatorname{Coh}([X^n/S_n]),\]
    is an equivalence of categories.
    Its transpose $\operatorname{BKR}^\top=q_\ast\circ \mathbf Lp^\ast$ is the transform considered by Krug \cite{krug}.
    Note that $q_\ast:\operatorname{Coh}([\tilde{X}_n/S_n])\to\operatorname{Coh}(X^{[n]})$ is already exact, see e.g.~\cite[Example 12.9]{goodmodulispace}.
    Viewed in terms of $S_n$-equivariant sheaves, $q_\ast$ takes a coherent sheaf to the $S_n$-invariant component of its pushforward via the quotient map $\tilde{X}_n\to X^{[n]}$.

    Now suppose $X\to C$ is an elliptic surface.
    Write $B=[C^n/S_n]$, which receives proper flat morphisms from $[X^n/S_n]$ as well as $[\tilde{X}_n/S_n]$.
    Note that $p$ is then a $B$-morphism.

    For a morphism $f$, we denote by $\operatorname{graph}_f$ the structure sheaf of its graph.
    \begin{lemma}
        Suppose $M'$ is integral and Gorenstein, $M'\to B$ is proper flat, and $\mathcal Q\in\operatorname{Coh}(M'\times_B [X^n/S_n])$ is maximal Cohen-Macaulay.
        Then $\operatorname{graph}_p\circ \mathcal Q$, viewed as a complex over $M'\times_B[\tilde{X}_n/S_n]$, is a maximal Cohen-Macaulay sheaf.
    \end{lemma}
    \begin{proof}
        Denote by
        \[
        \begin{tikzcd}
            &M'\times_B{[}\tilde{X}_n/S_n{]}\arrow{d}{i=\operatorname{id}\times (p, \operatorname{id})}&\\
            M'\times_B{[}X^n/S_n{]} & M'\times_B{[}X^n/S_n{]}\times_B{[}\tilde{X}_n/S_n{]}\arrow{r}{\operatorname{pr_{23}}}\arrow{d}{\operatorname{pr_{13}}}\arrow[swap]{l}{\operatorname{pr_{12}}} & {[}X^n/S_n{]}\times_B{[}\tilde{X}_n/S_n{]}\\
            & M'\times_B{[}\tilde{X}_n/S_n{]} &
        \end{tikzcd}
        \]
        the morphisms involved.

        We then have
        \[\operatorname{graph}_p\circ \mathcal Q=\mathbf R(\operatorname{pr}_{13})_\ast(\operatorname{pr}_{12}^\ast\mathcal Q\otimes^{\mathbf L}\operatorname{pr}_{23}^\ast\operatorname{graph}_p)=\mathbf R(\operatorname{pr}_{13})_\ast i_\ast\mathbf L i^\ast\operatorname{pr}_{12}^\ast\mathcal Q=\mathbf Lp_{M'}^\ast\mathcal Q,\]
        where $p_{M'}:M'\times_B[\tilde{X}_n/S_n]\to M'\times_B[X^n/S_n]$ is the base-change of $p$ to $M'$.
        Since $p$ has finite Tor-dimension (due to having a smooth target), $p_{M'}$ too has finite Tor-dimension.
        As $M'$, $[\tilde{X}_n/S_n]$, and $[X^n/S_n]$ are all Gorenstein, their respective flat morphisms to $B$ are Gorenstein, so both the domain and codomain of $p_{M'}$ are Gorenstein.

        Hence $\mathbf Lp_{M'}^\ast\mathcal Q$ is maximal Cohen-Macaulay by \cite[Lemma~2.3]{arinkin}.
    \end{proof}
    Now suppose $Y$ is another elliptic surface fibered over $C$.
    Write $\tilde{Y}_n$ for the corresponding isospectral Hilbert scheme, and $p_Y,q_Y$ for the projections.
    \begin{proposition}\label{BKR_congruence}
        Suppose $\mathcal Q\in\operatorname{Coh}([X^n/S_n]\times_B[Y^n/S_n])$ is maximal Cohen-Macaulay.
        Then $\mathcal O_{\tilde{Y}_n}^\top\circ\mathcal Q\circ\mathcal O_{\tilde{X}_n}$ is supported on $X^{[n]}\times_{C^{[n]}}Y^{[n]}$ and is a maximal Cohen-Macaulay sheaf there.
    \end{proposition}
    \begin{proof}
        Applying the preceding lemma on $\mathcal Q^\top$ shows that $\operatorname{graph}_p\circ \mathcal Q^\top$ is a maximal Cohen-Macaulay sheaf on $[Y^n/S_n]\times_B[\tilde{X}_n/S_n]$.
        Applying the same lemma again on $(\operatorname{graph}_p\circ \mathcal Q^\top)^\top$ shows that
        \[\mathcal Q'=\operatorname{graph}_{p_Y}\circ(\operatorname{graph}_p\circ \mathcal Q^\top)^\top=\operatorname{graph}_{p_Y}\circ \mathcal Q\circ\operatorname{graph}_p^\top\]
        is a maximal Cohen-Macaulay sheaf on $[\tilde{X}_n/S_n]\times_B[\tilde{Y}_n/S_n]$.

        Now
        \[
        \mathcal O_{\tilde{Y}_n}^\top\circ\mathcal Q\circ\mathcal O_{\tilde{X}_n}=\operatorname{graph}_{q_Y}^\top\circ\mathcal Q'\circ\operatorname{graph}_q.
        \]
        The claim follows as $q,q_Y$ are finite flat, and that direct summands of maximal Cohen-Macaulay sheaves are still maximal Cohen-Macaulay.
    \end{proof}

    Suppose $\mathcal P\in\operatorname{Coh}(X\times_C Y)$ is maximal Cohen-Macaulay.
    Let us form the complex $\mathcal P^{\boxtimes n}=\mathcal P\boxtimes^{\mathbf L}\cdots\boxtimes^{\mathbf L}\mathcal P$ on $(X\times_CY)^n=X^n\times_{C^n}Y^n$.
    The pushforward of this to $X^n\times Y^n$ is the same as the $n$-th outer tensor of the pushforward of $\mathcal P$ to $X\times Y$.
    Indeed, we have:
    \begin{lemma}\label{closed_embeddings_outer}
        Suppose $i:Z\to M,i':Z'\to M'$ are closed embeddings, then $i_\ast A\boxtimes^{\mathbf L}(i')_\ast A'=(i\times i')_\ast(A\boxtimes^{\mathbf L}A')$.
    \end{lemma}
    \begin{proof}
        We have
        \begin{align*}
            i_\ast A\boxtimes^{\mathbf L}(i')_\ast A'&=(\operatorname{pr}_M^{M\times M'})^\ast i_\ast A\otimes^{\mathbf L}(\operatorname{pr}_{M'}^{M\times M'})^\ast (i')_\ast A'\\
            &=(i\times\operatorname{id}_M)_\ast(\operatorname{pr}_{Z}^{Z\times M'})^\ast A\otimes^{\mathbf L}(\operatorname{pr}_{M'}^{M\times M'})^\ast (i')_\ast A'\\
            &=(i\times\operatorname{id}_M)_\ast((\operatorname{pr}_{Z}^{Z\times M'})^\ast A\otimes^{\mathbf L}(i\times\operatorname{id}_M)^\ast(\operatorname{pr}_{M'}^{M\times M'})^\ast (i')_\ast A')\\
            &=(i\times\operatorname{id}_M)_\ast((\operatorname{pr}_{Z}^{Z\times M'})^\ast A\otimes^{\mathbf L}(\operatorname{pr}_{M'}^{Z\times M'})^\ast (i')_\ast A')\\
            &=(i\times\operatorname{id}_M)_\ast((\operatorname{pr}_{Z}^{Z\times M'})^\ast A\otimes^{\mathbf L}(\operatorname{id}_Z\times i')_\ast(\operatorname{pr}_{Z'}^{Z\times Z'})^\ast A')\\
            &=(i\times\operatorname{id}_M)_\ast(\operatorname{id}_Z\times i')_\ast((\operatorname{id}_Z\times i')^\ast(\operatorname{pr}_{Z}^{Z\times M'})^\ast A\otimes^{\mathbf L}(\operatorname{pr}_{Z'}^{Z\times Z'})^\ast A')\\
            &=(i\times i')_\ast(A\boxtimes^{\mathbf L}A').\qedhere
        \end{align*}
    \end{proof}
    \begin{proposition}\label{mcm_outer_tensor}
        Suppose $Z,Z'$ are Gorenstein and projective.
        Let $A,A'$ be maximal Cohen-Macaulay sheaves on $Z,Z'$.
        Then $A\boxtimes^{\mathbf L}A'\in D\operatorname{Coh}(Z\times Z')$ is a maximal Cohen-Macaulay sheaf.
    \end{proposition}
    \begin{proof}
        Choose smooth projective varieties $M,M'$ and closed embeddings $i:Z\to M,i':Z'\to M'$.
        Let $d,d'$ be the respective codimensions.

        Write $N=i_\ast A\boxtimes^{\mathbf L}(i')_\ast A'\in D^b\operatorname{Coh}(M\times M')$.
        We have $N=(i\times i')_\ast(A\boxtimes^{\mathbf L}A')$ by the preceding lemma.
        In particular, $\operatorname{codim}\operatorname{supp}(N)\ge d+d'$.

        It is also clear that $N\in D^{\le 0}(M\times M')$.
        As for its dual,
        \begin{align*}
            \mathbf R\underline{\operatorname{Hom}}(N,\mathcal O)&=\mathbf R\underline{\operatorname{Hom}}((\operatorname{pr}_M^{M\times M'})^\ast i_\ast A\otimes^{\mathbf L}(\operatorname{pr}_{M'}^{M\times M'})^\ast (i')_\ast A',\mathcal O)\\
            &=\mathbf R\underline{\operatorname{Hom}}((\operatorname{pr}_M^{M\times M'})^\ast i_\ast A,\mathbf R\underline{\operatorname{Hom}}((\operatorname{pr}_{M'}^{M\times M'})^\ast (i')_\ast A',\mathcal O))\\
            &=\mathbf R\underline{\operatorname{Hom}}((\operatorname{pr}_M^{M\times M'})^\ast i_\ast A,\mathcal O)\otimes^{\mathbf L}\mathbf R\underline{\operatorname{Hom}}((\operatorname{pr}_{M'}^{M\times M'})^\ast (i')_\ast A',\mathcal O)
        \end{align*}
        Write $A_{M'}=(\operatorname{pr}_{Z}^{Z\times M'})^\ast A$.
        It is maximal Cohen-Macaulay by \cite[Lemma~2.3]{arinkin}.
        Then
        \begin{align*}
            \mathbf R\underline{\operatorname{Hom}}((\operatorname{pr}_M^{M\times M'})^\ast i_\ast A,\omega_{M\times M'})&=\mathbf R\underline{\operatorname{Hom}}((i\times\operatorname{id}_M)_\ast A_{M'},\omega_{M\times M'})\\
            &=(i\times\operatorname{id}_M)_\ast\mathbf R\underline{\operatorname{Hom}}(A_{M'},(i\times\operatorname{id}_M)^!\omega_{M\times M'})\\
            &=(i\times\operatorname{id}_M)_\ast\mathbf R\underline{\operatorname{Hom}}(A_{M'},\omega_{Z\times M'})[-d].
        \end{align*}
        Therefore $\mathbf R\underline{\operatorname{Hom}}((\operatorname{pr}_M^{M\times M'})^\ast i_\ast A,\mathcal O)=\mathbf R\underline{\operatorname{Hom}}((\operatorname{pr}_M^{M\times M'})^\ast i_\ast A,\omega_{M\times M'})\otimes \omega_{M\times M'}^{\vee}$ is concentrated at degree $d$.
        Similarly, $\mathbf R\underline{\operatorname{Hom}}((\operatorname{pr}_{M'}^{M\times M'})^\ast (i')_\ast A',\mathcal O)$ is concentrated in degree $d'$.

        Combining these, we see that
        \[\mathbb DN=\mathbf R\underline{\operatorname{Hom}}(N,\mathcal O)\otimes\omega_{M\times M'}\in D^{\le d+d'}(M\times M').\]
        Hence $N$ is a Cohen-Macaulay sheaf of codimension $d+d'$ by \cite[Lemma 7.7]{arinkin}.
        This means that
        \begin{align*}
            (i\times i')_\ast\mathbb D(A\boxtimes^{\mathbf L}A')&=(i\times i')_\ast\mathbf R\underline{\operatorname{Hom}}(A\boxtimes^{\mathbf L}A',(i\times i')^!\omega_{M\times M'}[d+d'])\\
            &=\mathbf R\underline{\operatorname{Hom}}(N,\omega_{M\times M'})[d+d']=\mathbb DN[d+d']
        \end{align*}
        is concentrated in degree $0$.
        So $A\boxtimes^{\mathbf L}A'$ is maximal Cohen-Macaulay.
    \end{proof}
    \begin{corollary}
        The complex $\mathcal P^{\boxtimes n}$ is a maximal Cohen-Macaulay sheaf.
    \end{corollary}
    Suppose in addition that the Fourier-Mukai transform induced by $\mathcal P$ is an exact equivalence.
    Then $\mathcal P^{\boxtimes n}$ too induces an exact equivalence $D^b(X^n)\to D^b(Y^n)$.

    The work of \cite{ploog} shows we have more.
    \begin{definition}
        For a finite group $G$ acting on a variety $X$ and a subgroup $H\le G$, we set $\operatorname{ind}_H^G:D^b\operatorname{Coh}_H(X)\to D^b\operatorname{Coh}_G(X)$ and $\operatorname{res}_H^G:D^b\operatorname{Coh}_G(X)\to D^b\operatorname{Coh}_H(X)$ for the pushforward and pullback along the natural map $[X/H]\to [X/G]$.
    \end{definition}
    Practically, $\operatorname{res}_H^G$ is simply the forgetful map, and $\operatorname{ind}_H^G$ sends a $H$-linearized object $A$ to
    \[
    \operatorname{ind}_H^GA=\bigoplus_{g\in H\backslash G}g^\ast A,
    \]
    equipped with the apparent $G$-linearization.

    Now, we have a natural diagonal $S_n$-linearization $\rho_{\rm perm}$ of $\mathcal P^{\boxtimes n}$ by permuting the coordinates.
    We may then consider $\operatorname{ind}_{S_n}^{S_n\times S_n}(\mathcal P^{\boxtimes n},\rho_{\rm perm})$ which is a maximal Cohen-Macaulay sheaf on $[X^n\times_{C^n}Y^n/S_n\times S_n]$.

    By \cite[Lemma 5(5)]{ploog}, the kernel $\operatorname{ind}_{S_n}^{S_n\times S_n}(\mathcal P^{\boxtimes n},\rho_{\rm perm})$ induces an equivalence $\mathcal F^{\boxtimes n}:D^b([X^n/S_n])\to D^b([Y^n/S_n])$.

    \begin{remark}\label{powerfm}
        The transform $\mathcal F^{\boxtimes n}$ can be evaluated without computing the induction.
        Suppose $G$ acts on schemes $M,N$ and $\mathcal Q$ is a complex on $M\times N$ linearized by the diagonal action of $G=G_\Delta$.
        Then
        \begin{align*}
            (\operatorname{pr}_2)_\ast(\operatorname{ind}_{G_\Delta}^{G\times G}\mathcal Q\otimes \operatorname{pr}_1^\ast A)&=(\operatorname{pr}_2)_\ast\operatorname{ind}_{G_\Delta}^{G\times G}(\mathcal Q\otimes \operatorname{res}_{G_\Delta}^{G\times G}\operatorname{pr}_1^\ast A)\\
            &=(\operatorname{pr}_2^\Delta)_\ast(\mathcal Q\otimes (\operatorname{pr}_1^\Delta)^\ast A),
        \end{align*}
        where $\operatorname{pr}_1^\Delta:[(M\times N)/G_\Delta]\to [M/G]$ and $\operatorname{pr}_2^\Delta:[(M\times N)/G_\Delta]\to [N/G]$.
    \end{remark}

    We can then compose these transforms to obtain another equivalence with Cohen-Macaulay kernel.
    \begin{corollary}\label{BKR_finalform}
        Suppose $X,Y$ are elliptic surfaces over a common base $C$, and $\mathcal P\in \operatorname{Coh}(X\times_CY)$ is a maximal Cohen-Macaulay sheaf inducing an exact equivalence $D^b(X)\to D^b(Y)$.
        Then the kernel of the composition
        \[
        \begin{tikzcd}
        D^b(X^{[n]})\arrow{r}{\operatorname{BKR}}&D^b([X^n/S_n])\arrow{r}{\mathcal F^{\boxtimes n}}&D^b([Y^n/S_n])\arrow{r}{\operatorname{BKR}^\top}&D^b(Y^{[n]})
        \end{tikzcd}
        \]
        is a maximal Cohen-Macaulay sheaf on $X^{[n]}\times_{C^{[n]}}Y^{[n]}$.
    \end{corollary}
    \begin{proof}
        Proposition \ref{BKR_congruence}.
    \end{proof}

    Note that this new kernel is immediately flat over both factors, since we have:
    \begin{lemma}
        Suppose $P\to M$ is a flat Gorenstein morphism between integral schemes with $M$ smooth.
        Then any maximal Cohen-Macaulay $\mathcal Q\in\operatorname{Coh}(P)$ is flat over $M$.
    \end{lemma}
    \begin{proof}
        For any $x\in M$, the inclusion of the fiber at $x$ under the projection $P\to M$ is a flat base-change of the inclusion $x\hookrightarrow M$, which has finite tor-dimension as $M$ is smooth.
        The result then follows from \cite[Lemma 2.3]{arinkin} and \cite[Lemma 3.31]{HuyFM}.
    \end{proof}
    \section{Theorem of the square}\label{ss:tots}
    We will be interested in the case where $X=Y$ and $\mathcal P$ is the kernel in Theorem \ref{surfacegroupscheme}, normalized so that $\Phi_{\mathcal P}(s_\ast\mathcal O)=\mathcal O$.
    Write $\mathcal P^{\operatorname{BKR}}_n$ for the kernel as in Corollary \ref{BKR_finalform}.
    Note that since $\mathcal P^\top=\mathcal P$, we have $(\mathcal P_n^{\rm BKR})^\top=\mathcal P_n^{\rm BKR}$ by Lemma \ref{transpose}.

    Denote by $\mathcal P^{\rm BKR,\circ}_n$ the restriction $\mathcal P^{\rm BKR}_n|_{X^{[n],\pi}\times_{C^{[n]}} X^{[n]}}$.
    Since $X^{[n],\pi}\to C^{[n]}$ is a group scheme, it is smooth.
    Hence $\mathcal P^{\rm BKR,\circ}_n$ is locally free.
    In fact, it has rank $1$:
    \begin{proposition}\label{section_BKR}
        The restriction of $\mathcal P^{\rm BKR}_n$ to $s^{[n]}(C^{[n]})\times_{C^{[n]}}X^{[n]}$ is the trivial line bundle.
    \end{proposition}
    \begin{proof}
        To simplify notation, we write $\mathcal P$ in place of its pushforward to $X\times X$ during this proof.

        It suffices to show that $\operatorname{BKR}^\top\circ\mathcal F^{\boxtimes n}\circ\operatorname{BKR}$ sends $\mathcal O_{s^{[n]}(C^{[n]})}$ to $\mathcal O_{X^{[n]}}$.
    
        By Remark \ref{BKRcurves}, we have $\operatorname{BKR}(\mathcal O_{s^{[n]}(C^{[n]})})=\mathcal O_{s^n(C^n)}$ where $\mathcal O_{s^n(C^n)}=(s^n)_\ast\mathcal O_{C^n}$ is linearized with the $S_n$-action on $C^n$ via the $S_n$-equivariant closed immersion $s^n:C^n\hookrightarrow X^n$.

        Lemma \ref{closed_embeddings_outer}, on the other hand, tells us that $(s^n)_\ast\mathcal O\cong s_\ast\mathcal O\boxtimes^{\mathbf L}\cdots\boxtimes^{\mathbf L} s_\ast\mathcal O$.
        This is an isomorphism of $S_n$-linearized sheaves if we linearize the right-hand side by permuting the tensor factors.

        Let $\operatorname{pr}_1^\Delta,\operatorname{pr}_2^\Delta:[(X^n\times X^n)/(S_n)_\Delta]\to [X^n/S_n]$ be the projections as in Remark \ref{powerfm}.
        Then
        \[
        \mathcal P^{\boxtimes n}\otimes^{\mathbf L}(\operatorname{pr}_1^\Delta)^\ast(s_\ast\mathcal O)^{\boxtimes n}=\bigotimes_{i=1}^n{}^{\mathbf L}(q_i^\ast\mathcal P\otimes^{\mathbf L}\operatorname{pr}_1^\ast p_i^\ast s_\ast\mathcal O),
        \]
        where $p_i:X^n\to X,q_i:X^n\times X^n\to X\times X$ are the projections used to build the respective outer tensors.
        Note that this is an $(S_n)_\Delta$-equivariant isomorphism with the right-hand side linearized by permuting the tensor factors.

        Now $p_i\circ\operatorname{pr}_1=r_1\circ q_i$ where $r_i:X\times X\to X$.
        Our normalization $\Phi_{\mathcal P}(s_\ast\mathcal O)=\mathcal O$ means that $\mathcal P\otimes^{\mathbf L}r_1^\ast s_\ast\mathcal O=\mathcal O_{s(C)\times_CX}$.
        Using the same argument as before, we also have $q_1^\ast\mathcal O_{s(C)\times_CX}\otimes^{\mathbf L}\cdots\otimes^{\mathbf L}q_n^\ast\mathcal O_{s(C)\times_CX}\cong \mathcal O_{s^n(C^n)\times_{C^n}X^n}$ as $(S_n)_\Delta$-sheaves.

        Hence $\mathcal F^{\boxtimes n}((s^n)_\ast\mathcal O_{C^n})=\mathcal O_{[X^n/S_n]}$ (i.e.~$\mathcal O_{X^n}$ with the natural $S_n$-action).
        Now we apply $\operatorname{BKR}^\top$.
        The pullback of $\mathcal O_{[X^n/S_n]}$ to $\tilde{X}_n$ is just $\mathcal O_{[\tilde{X}_n/S_n]}$, and its pushforward to the quotient, namely $X^{[n]}$, is then $\mathcal O_{X^{[n]}}$.
    \end{proof}
    \begin{remark}
        Here, and in later parts, evaluating the image of an outer tensor under $\mathcal F^{\boxtimes n}$ (and analogues thereof) follows from well-known calculations (e.g.~\cite[Exercise 5.13]{HuyFM}) if we ignore the linearization.
        An explicit computation is shown here to keep track of the respective linearizations.
    \end{remark}
    \begin{proposition}[Theorem of the square]\label{tots}
        The identity
        \[(\mu\times\operatorname{id})^\ast\mathcal P^{\rm BKR}_n=\operatorname{pr}_{13}^\ast\mathcal P^{\rm BKR,\circ}_n\otimes\operatorname{pr}_{23}^\ast\mathcal P^{\rm BKR}_n\]
        holds on $X^{[n],\pi}\times_{C^{[n]}}X^{[n]}\times_{C^{[n]}}X^{[n]}$.
    \end{proposition}
    Both sides are maximal Cohen-Macaulay, therefore it suffices to show this on an open set whose complement has codimension at least $2$.
    Let us now construct it.

    Consider the open set $U\subset C^{[n]}$ consisting of $n_1p_1+\cdots+n_rp_r$ such that either
    \begin{enumerate}
        \item $n_k=1$ for all $k$, or
        \item there is exactly one $k$ such that $n_k=2$, with $n_j=1$ for $j\neq k$, and $X_{p_i}$ is smooth for all $i$.
    \end{enumerate}
    Note that $C^{[n]}\setminus U$ has codimension $2$.

    Let $X^{[n],\pi,\sharp}_U\subset X_U^{[n]}$ be the subset that consists of finite subschemes that are either reduced, or the (necessarily unique) non-reduced component is transversal to the (necessarily smooth) fiber to which its reduction belongs.

    Note that $X_U^{[n]}\setminus X_U^{[n],\pi,\sharp}$ consists of finite subschemes of the form $Z\sqcup Z_3\sqcup\cdots\sqcup Z_n$, with $Z_i$ reduced for all $i=3,\ldots,n$, and $Z$ a non-reduced length $2$ subscheme of $X_{\pi(Z_{\rm red})}$ (which is smooth by definition of $U$).
    This is closed and has dimension $2(n-2)+2=2n-2$, i.e.~codimension $2$ in $X_U^{[n]}$.

    We shall show the identity in Proposition \ref{tots} after restricting to
    \[X^{[n],\pi}_U\times_UX^{[n],\pi,\sharp}_U\times_UX^{[n]}_U,\]
    whose complement in $X^{[n],\pi}\times_{C^{[n]}}X^{[n]}\times_{C^{[n]}}X^{[n]}$ has codimension $2$, as discussed.

    Our next claim is that it suffices to show that the restrictions of the two sides to the fibers over any closed point $(Z,Z')\in X^{[n],\pi}_U\times_UX^{[n],\pi,\sharp}_U$ on the first two factors coincide.
    To do this, we make use of the following generalization of the classical seesaw principle to simple sheaves.
    \begin{proposition}[Seesaw principle for simple sheaves]\label{seesaw}
        Let $T\to B$ be a morphism between reduced varieties, and let $M\to B$ be projective and flat.
        Suppose we have coherent sheaves $\mathcal F,\mathcal G$ on $M\times_BT$ flat over $T$ such that, for all $k$-points $t$ of $T$,
        \begin{enumerate}
            \item $\mathcal F_t$ is simple, and
            \item there exists an isomorphism $\mathcal F_t\cong\mathcal G_t$.
        \end{enumerate}
        Then $\mathcal F,\mathcal G$ are in the same $\operatorname{Pic}(T)$-orbit.
    \end{proposition}
    \begin{proof}
        Let us first analyze the analytifications $\mathcal F^{\rm an}$, $\mathcal G^{\rm an}$.

        Consider the functor $\operatorname{Spl}_{M^{\rm an}/B^{\rm an}}$ that sends an analytic space $T'$ over $B^{\rm an}$ to the set of $\operatorname{Pic}(T')$-orbits of coherent sheaves on $M^{\rm an}\times_{B^{\rm an}}T'$, flat over $T'$, whose restrictions to the fibers are simple.

        Write $\operatorname{Spl}_{M^{\rm an}/B^{\rm an}}^+$ for its sheafification.
        As noted in \cite[(6.7)]{analyticsimplesheaves}, the presheaf $\operatorname{Spl}_{M^{\rm an}/B^{\rm an}}$ satisfies the separation axiom of a sheaf, so the sheafification map induces an inclusion $\operatorname{Spl}_{M^{\rm an}/B^{\rm an}}\hookrightarrow \operatorname{Spl}_{M^{\rm an}/B^{\rm an}}^+$.

        We claim that the maps $\eta_{\mathcal F^{\rm an}},\eta_{\mathcal G^{\rm an}}:T^{\rm an}\to\operatorname{Spl}_{M^{\rm an}/B^{\rm an}}^+$ defined by $\mathcal F^{\rm an},\mathcal G^{\rm an}$ coincide.
        Indeed, by \cite[(6.4)]{analyticsimplesheaves}, $\operatorname{Spl}_{M^{\rm an}/B^{\rm an}}^+$ is representable by an analytic space.
        As the diagonal of an analytic space is always an immersion, and $\eta_{\mathcal F^{\rm an}},\eta_{\mathcal G^{\rm an}}$ agree on closed points by assumption, they would have to be equal.

        Therefore $\mathcal F^{\rm an},\mathcal G^{\rm an}$ are in the same $\operatorname{Pic}(T^{\rm an})$-orbit.
        Suppose $\mathcal G^{\rm an}=\mathcal F^{\rm an}\otimes p^\ast L$ where $p$ is (the analytification of) the projection to $T$ and $L\in\operatorname{Pic}(T^{\rm an})$.
        Then $\underline{\operatorname{Hom}}(\mathcal F^{\rm an},\mathcal G^{\rm an})=\underline{\operatorname{Hom}}(\mathcal F^{\rm an},\mathcal F^{\rm an})\otimes p^\ast L$.
        This means that $p_\ast\underline{\operatorname{Hom}}(\mathcal F^{\rm an},\mathcal G^{\rm an})=L$ and the natural morphism
        \[\mathcal F^{\rm an}\otimes p^\ast p_\ast\underline{\operatorname{Hom}}(\mathcal F^{\rm an},\mathcal G^{\rm an})\to\mathcal G^{\rm an}\]
        is an isomorphism.

        Now let us consider the algebraic situation.
        We have by \cite[Proposition 21]{gaga} and \cite[Theorem 7.1]{goodgaga} that
        \[(p_\ast\underline{\operatorname{Hom}}(\mathcal F,\mathcal G))^{\rm an}=p_\ast\underline{\operatorname{Hom}}(\mathcal F^{\rm an},\mathcal G^{\rm an})\]
        is a line bundle.
        This means that $p_\ast\underline{\operatorname{Hom}}(\mathcal F,\mathcal G)$ must too be a line bundle, and the map
        \[\mathcal F\otimes p^\ast p_\ast\underline{\operatorname{Hom}}(\mathcal F,\mathcal G)\to\mathcal G\]
        is an isomorphism since its analytification is.
    \end{proof}
    \begin{remark}
        One could obtain an algebraic proof if the moduli space considered in \cite[Theorem 7.4]{altklei} were locally separated, i.e.~the diagonal is an immersion.
        We are unable to find a reference for this, so our argument used instead the analytic moduli space in \cite{analyticsimplesheaves}.

        A version of this seesaw principle has been proven in \cite[Lemma 5.5]{mrv2}.
        Like in our argument, their proof utilized the fact that $p_\ast\underline{\operatorname{Hom}}(\mathcal F,\mathcal G)$ is a line bundle.
        We are not aware of an algebraic way to justify this.
    \end{remark}
    \begin{corollary}\label{fibercheck_tots}
        Suppose that for any $(Z,Z')\in X_U^{[n],\pi}\times_UX_U^{[n],\pi,\sharp}$, we have
        \[
        \Phi_{\mathcal P_n^{\rm BKR}}(k(\mu(Z,Z')))=\Phi_{\mathcal P_n^{\rm BKR}}(k(Z))\otimes^{\rm und}\Phi_{\mathcal P_n^{\rm BKR}}(k(Z')).
        \]
        Then Proposition \ref{tots} is true.
    \end{corollary}
    \begin{proof}
        We apply the preceding proposition to $B=U$, $T=X_U^{[n],\pi}\times_UX_U^{[n],\pi,\sharp}$, and $M=X_U^{[n]}$.
        The two coherent sheaves we compare will be the two sides of Proposition \ref{tots} restricted to $X^{[n],\pi}_U\times_UX^{[n],\pi,\sharp}_U\times_UX^{[n]}_U$.

        Let $W=\pi^{[n]}(Z)=\pi^{[n]}(Z')$.
        Then the fiber of $(\mu\times \operatorname{id})^\ast\mathcal P_n^{\rm BKR}$ at $(Z,Z')$ is exactly the fiber of $\mathcal P_n^{\rm BKR}$ at $\mu(Z,Z')$.
        This is a coherent sheaf on $X^{[n]}_{W}$ whose pushforward to $X^{[n]}$ is exactly $\Phi_{\mathcal P_n^{\rm BKR}}(k(\mu(Z,Z')))$.
        Note also that it is simple as it is the image of a skyscraper sheaf under an equivalence.

        On the other hand, the fiber of $\operatorname{pr}_{13}^\ast\mathcal P^{\rm BKR,\circ}_n\otimes\operatorname{pr}_{23}^\ast\mathcal P^{\rm BKR}_n$ at $(Z,Z')$ is the tensor product of the fibers of $\mathcal P_n^{\rm BKR}$ at $Z$ and $Z'$, both sheaves on $X_W^{[n]}$.
        The pushforward of this to $X^{[n]}$ is then the right-hand side of the equality.

        By the preceding proposition, the two sides differ by a line bundle pulled back from $X^{[n],\pi}_U\times_UX^{[n],\pi,\sharp}_U$.
        Now note that both sides restrict to the structure sheaf on $X^{[n],\pi}_U\times_UX^{[n],\pi,\sharp}_U\times_Us^{[n]}(U)$ by Proposition \ref{section_BKR}.
        The claim then follows from the preceding proposition.
    \end{proof}
    The situation is now reduced to understanding the Fourier-Mukai images of certain skyscraper sheaves under $\mathcal P_n^{\rm BKR}$.

    Let us first analyze their images under $\operatorname{BKR}$.
    From now on, when we write $S_k$ as a subgroup of $S_n$ for $k\le n$, we will always mean the inclusion induced by the inclusion of the letters $\{1,\ldots,k\}\hookrightarrow \{1,\ldots, n\}$.

    By definition, $\operatorname{BKR}(k(Z))$ is the structure sheaf of the $S_n$-cluster $\operatorname{BKR}_Z$ (on the $n$-fold product) corresponding to $Z$.
    \begin{lemma}\label{cluster_projection}
        Suppose $Z$ is a finite subscheme of length $n$ in a smooth projective variety $Y$ of dimension at most $2$.
        Then for any $i$, $\operatorname{BKR}_Z$ is a closed subscheme of $\operatorname{pr}_i^{-1}(Z)$.
    \end{lemma}
    \begin{proof}
        It suffices to show that $\tilde{Y}_n$ is a closed subscheme of the preimage of the universal subscheme on $Y\times Y^{[n]}$ via the $i$-th projection $Y^n\times Y^{[n]}\to Y\times Y^{[n]}$.
        The inclusion on $k$-points is clear.
        But $\tilde{Y}_n$ is also reduced, so this is an inclusion of closed subschemes.
    \end{proof}
    \begin{corollary}\label{proj_isom_mult2}
        Suppose $Z$ has length $2$, then each of the projections factors through an isomorphism $\operatorname{BKR}_Z\to Z$.
    \end{corollary}
    \begin{proof}
        By the preceding lemma, the scheme-theoretic image of $\operatorname{BKR}_Z$ is contained in $Z$ and therefore either has length $1$ or $2$.
        In the former case, say the image is $k(x)$, we must have by $S_2$-invariance that the scheme-theoretic image of $\operatorname{BKR}_Z$ under the second projection is also $k(x)$.
        This, however, would mean that $\operatorname{BKR}_Z$ is a subscheme of $k(x)\boxtimes k(x)=k(x,x)$, which is a contradiction as $k(x,x)$ has length $1$.

        Hence the image is exactly $Z$.
        But $\operatorname{BKR}_Z$ has length $2$ as well, so this is an isomorphism.
    \end{proof}
    \begin{lemma}
        Suppose $Z$ is a length $n'\le n$ subscheme of a smooth projective variety $Y$ of dimension at most $2$, and $x_{n'+1},\ldots,x_n$ are distinct points in $Y\setminus Z$.
        Then we have an $S_{n'}$-equivariant isomorphism
        \[\operatorname{BKR}(k(Z))\boxtimes^{\mathbf L} k(x_{n'+1})\boxtimes^{\mathbf L}\cdots\boxtimes^{\mathbf L} k(x_n)\cong\mathcal O_{\operatorname{BKR}_Z\times \{x_{n'+1}\}\times \cdots\times\{x_n\}}.\]
        Moreover, $\operatorname{ind}_{S_{n'}}^{S_n}\mathcal O_{\operatorname{BKR}_Z\times \{x_{n'+1}\}\times \cdots\times\{x_n\}}$ is an $S_n$-cluster.
    \end{lemma}
    \begin{proof}
        The isomorphism is Lemma \ref{closed_embeddings_outer}.
        Let us show that the induction is an $S_n$-cluster.

        Note first that the induction, after forgetting the linearization, is the structure sheaf of a length $n!$ subscheme $\mathfrak Z$, as it is the direct sum of $n!/(n')!$ structure sheaves of length $(n')!$ subschemes with disjoint supports.
        We thus have a surjection $\mathcal O_{Y^n}\to\mathcal O_{\mathfrak Z}$.

        This surjection is the adjunction to the $S_{n'}$-equivariant surjection $\operatorname{res}_{S_{n'}}^{S_n}\mathcal O_{Y^n}\to \mathcal O_{\operatorname{BKR}_Z\times \{x_{n'+1}\}\times \cdots\times\{x_n\}}$ extended from $\mathcal O_{Y^{n'}}\to\mathcal O_{\operatorname{BKR}_Z}$.
        In particular, it is $S_n$-equivariant.
        So the linearization on the induction coincides with that on $\mathcal O_{\mathfrak Z}$.

        Since $\Gamma(\mathcal O_{\operatorname{BKR}_Z\times \{x_{n'+1}\}\times \cdots\times\{x_n\}})=\Gamma(\mathcal O_{\operatorname{BKR}_Z})$ is the regular $S_{n'}$-representation, $\Gamma(\mathcal O_{\mathfrak Z})\cong\operatorname{ind}_{S_{n'}}^{S_n}\Gamma(\mathcal O_{\operatorname{BKR}_Z\times \{x_{n'+1}\}\times \cdots\times\{x_n\}})$ is the regular $S_n$-representation.
    \end{proof}
    \begin{corollary}
        Suppose $Z$ is a length $2$ subscheme of $Y$ and $x_3,\ldots,x_n\in Y\setminus Z$ are distinct.
        Write $Z^+=Z\sqcup x_3\sqcup\cdots\sqcup x_n$.
        Then
        \[\operatorname{BKR}(k(Z^+))=\operatorname{ind}_{S_2}^{S_n}(\operatorname{BKR}(k(Z))\boxtimes^{\mathbf L} k(x_3)\boxtimes^{\mathbf L}\cdots\boxtimes^{\mathbf L} k(x_n)).\]
        If furthermore $Z=x_1\sqcup x_2$ is reduced, then $\operatorname{BKR}(k(Z^+))=\operatorname{ind}_{S_1}^{S_n}k(x_1,\ldots,x_n)=\operatorname{ind}_{S_1}^{S_n}(k(x_1)\boxtimes^{\mathbf L}\cdots\boxtimes^{\mathbf L} k(x_n))$.
    \end{corollary}
    \begin{proof}
        By the preceding lemma, the claimed forms of $\operatorname{BKR}$-images are indeed $S_n$-clusters.
        The only thing left to do is to identify the preimages of these clusters.
        But by Corollary \ref{proj_isom_mult2} the scheme-theoretic image of each of these clusters under any of the projections is exactly $Z^+$, hence by Lemma \ref{cluster_projection} they can only correspond to $Z^+$.
    \end{proof}
    For the type of subschemes we need to consider, this allows us to reduce the situation to the case $n=2$.
    More precisely, we can utilize the following:
    \begin{lemma}\label{induction_FM}
        Let $M,M'$ be smooth projective varieties with actions by a finite group $G$.
        Let $H$ be a subgroup of $G$.

        Then for any $\mathcal Q\in D^b\operatorname{Coh}([(M\times M')/G_\Delta])$ and $A\in D^b\operatorname{Coh}([M/H])$, we have
        \[\Phi_{\mathcal Q}(\operatorname{ind}_H^GA)= \operatorname{ind}_H^G\Phi_{\operatorname{res}^{G_\Delta}_{H_\Delta}\mathcal Q}(A).\]
    \end{lemma}
    \begin{proof}
        Write $\operatorname{pr}_i^G$ for the projections from $[(M\times M')/G_\Delta]$ and $\operatorname{pr}_i^H$ for those from $[(M\times M')/H_\Delta]$.
        Note first that $(\operatorname{pr}_1^G)^\ast\operatorname{ind}_H^G=\operatorname{ind}_{H_\Delta}^{G_\Delta}(\operatorname{pr}_1^H)^\ast$.
        Indeed, this is true for any $G$-equivariant map between $G$-varieties.
        We therefore have
        \begin{align*}
            (\operatorname{pr}_2^G)_\ast(\mathcal Q\otimes (\operatorname{pr}_1^G)^\ast \operatorname{ind}_H^GA)&=(\operatorname{pr}_2^G)_\ast(\mathcal Q\otimes\operatorname{ind}_{H_\Delta}^{G_\Delta}(\operatorname{pr}_1^H)^\ast A)\\
            &=(\operatorname{pr}_2^G)_\ast\operatorname{ind}_{H_\Delta}^{G_\Delta}(\operatorname{res}_{H_\Delta}^{G_\Delta}\mathcal Q\otimes (\operatorname{pr}_1^H)^\ast A)\\
            &=\operatorname{ind}_H^G(\operatorname{pr}_2^H)_\ast(\operatorname{res}_{H_\Delta}^{G_\Delta}\mathcal Q\otimes (\operatorname{pr}_1^H)^\ast A)\qedhere
        \end{align*}
    \end{proof}
    Note now that $\operatorname{res}_{S_2}^{S_n}(\mathcal P^{\boxtimes n},\rho_{\operatorname{perm}})=(\mathcal P^{\boxtimes 2},\rho_{\operatorname{perm}})\boxtimes^{\mathbf L}\mathcal P^{\boxtimes (n-2)}$.
    So it remains to understand the $n=2$ case.

    Suppose $W=2p\hookrightarrow C$ for some $p\in C$ such that $X_p$ is smooth.
    By Corollary \ref{proj_isom_mult2}, each of the projections gives isomorphisms $\operatorname{BKR}_W\to W$.
    Hence the projections restrict to $r_i:(X^2)_{\operatorname{BKR}_W}\to X_W$.
    \[\begin{tikzcd}
        X_W\dar&(X^2)_{\operatorname{BKR}_W}\dar\rar\lar&X_W\dar\\
        W&\operatorname{BKR}_W\rar{\cong}\lar[swap]{\cong}&W
    \end{tikzcd}.\]
    \begin{lemma}
        The map $(r_1,r_2):(X^2)_{\operatorname{BKR}_W}\to X_W\times_{\operatorname{BKR}_W}X_W$ is an isomorphism, where each $X_W$ is considered a $\operatorname{BKR}_W$-scheme by inverting the isomorphisms $\operatorname{BKR}_W\to W$ on each side.
    \end{lemma}
    \begin{proof}
        The map $(r_1,r_2)$ is a set-theoretic bijection.
        It remains to check that the induced maps on local rings are isomorphisms.
        This can be checked after passing to strict Henselizations at the corresponding geometric points.

        As $\pi$ is smooth near $p$, the strict Henselian local model is the base change of $\pi:\mathbb A^2\to\mathbb A^1,(x,y)\mapsto x$ with $W=\mathbb V(x^2)$.
        In this model
        \[
        \operatorname{BKR}_W=\operatorname{Spec} k[x_1,x_2]/(x_1+x_2,(x_1-x_2)^2).
        \]
        Indeed, this is an $S_2$-cluster and it projects to exactly $W$.

        So the projections $W\leftarrow\operatorname{BKR}_W\rightarrow W$ can be represented by $k[x_2]/(x_2^2)\rightarrow k[x_1,x_2]/(x_1+x_2,(x_1-x_2)^2)\leftarrow k[x_1]/(x_1^2)$.
        Note that these are isomorphisms, which identify $x_1\leftrightarrow -x_2$.
        Hence $X_W\rightarrow\operatorname{BKR}_W\leftarrow X_W$ can be represented by $k[x_1,y_1]/(x_1^2)\leftarrow k[x_1]/(x_1^2)\rightarrow k[-x_1,y_2]/(x_1^2)$.

        Thus $(r_1,r_2)$ is represented by the map \[k[x_1,y_1,y_2]/(x_1^2)\to k[x_1,x_2,y_1,y_2]/(x_1+x_2,(x_1-x_2)^2)\] which is an isomorphism.
    \end{proof}
    \begin{proposition}
        Let $Z\in X^{[2],\pi}_U$ and suppose the corresponding length $2$ subscheme is connected.
        Write $W=\pi^{[2]}(Z)$.
        Then we have an $S_2$-equivariant isomorphism
        \[\operatorname{BKR}(k(Z))\cong (i_{(X^2)_{\operatorname{BKR}_W}})_\ast(r_1^\ast\mathcal O_Z\otimes^{\mathbf L}r_2^\ast\mathcal O_Z),\]
        with the latter linearized by swapping the factors.
    \end{proposition}
    \begin{proof}
        Let us first show that the right-hand side is the structure sheaf of a length $2$ subscheme $Z'$.
        This follows from the preceding lemma.
        Indeed, if $i:r_1^{-1}(Z)\to (X^2)_{\operatorname{BKR}_Z}$ is the inclusion, then $r_2\circ i$ is an isomorphism (this is because $Z\to W$ is an isomorphism, which is a section of $X_W\to W\cong\operatorname{BKR}_W$).
        So
        \[r_1^\ast\mathcal O_Z\otimes^{\mathbf L}r_2^\ast\mathcal O_Z=i_\ast\mathbf Li^\ast r_2^\ast\mathcal O_Z=i_\ast (r_2\circ i)^\ast\mathcal O_Z,\]
        which is the structure sheaf of the closed subscheme corresponding to $Z$ under the isomorphism $r_2\circ i$.
        In particular, it has length $2$.

        Note also that the $S_2$-action on $Z'$ is precisely the $S_2$-action on
        \[r_1^{-1}(Z)\times_{(X^2)_{\operatorname{BKR}_W}}r_2^{-1}(Z)=r_1^{-1}(Z)\times_{(X^2)_{\operatorname{BKR}_W}}[-1]^\ast r_1^{-1}(Z),\]
        by swapping the product factors, where $[-1]$ denotes the non-identity element in $S_2$.
        Hence the isomorphism $\mathcal O_{Z'}\cong r_1^\ast\mathcal O_Z\otimes^{\mathbf L}r_2^\ast\mathcal O_Z$ is $S_2$-equivariant.

        Since $Z\subset X_W$, $\operatorname{BKR}_Z$ must be a closed subscheme of the preimage of $\operatorname{BKR}_W$ under the projection to $C^2$.
        But we also know that it is a closed subscheme of $Z$ under each of the projections to $X$.
        This means that $\operatorname{BKR}_Z\hookrightarrow Z'$.
        But both have length $2$, hence they are the same.
    \end{proof}
    \begin{remark}
        Note that it is important that $Z\in X_U^{[2],\pi}$ here.
        The proposition is not true in general for length $2$ subschemes.
    \end{remark}
    \begin{corollary}
        Let $Z,W$ be as in the preceding proposition, then
        \[\mathcal F^{\boxtimes 2}(\operatorname{BKR}(k(Z)))=(i_{(X^2)_{\operatorname{BKR}_W}})_\ast (r_1^\ast L\otimes r_2^\ast L),\]
        where $L$ is the line bundle on $X_W$ such that $(i_{X_W})_\ast L=\mathcal F(\mathcal O_Z)$.
        The linearization on the right-hand side is by swapping the factors.
    \end{corollary}
    \begin{proof}
        We have a commutative diagram
        \[\begin{tikzcd}
            (X^2)_{\operatorname{BKR}_W}\dar{r_1}[swap]{r_2}&\lar[swap]{p_1^{\operatorname{BKR}_W}} (X^2)_{\operatorname{BKR}_W}\times_{\operatorname{BKR}_W}(X^2)_{\operatorname{BKR}_W}\rar{i_{\operatorname{BKR}_W}}\dar{(r_1,r_1)}[swap]{(r_2,r_2)}&X^2\times_{C^2}X^2\dar{\operatorname{pr}_{13}}[swap]{\operatorname{pr}_{24}}\\
            X_W&\lar[swap]{p_1^W} X_W\times_WX_W\rar{i_W}&X\times_CX
        \end{tikzcd},\]
        where the horizontal arrows are the natural inclusions and first projections, and vertical arrows are given by first (resp.~second) projections on each factor.

        We then have, by the preceding proposition:
        \begin{align*}
            \mathbf Li_{\operatorname{BKR}_W}^\ast\mathcal P^{\boxtimes 2}\otimes^{\mathbf L} (p_1^{\operatorname{BKR}_W})^\ast\operatorname{BKR}(k(Z))=(r_1,r_1)^\ast\mathcal R\otimes^{\mathbf L}(r_2,r_2)^\ast\mathcal R,
        \end{align*}
        where $\mathcal R=\mathbf Li_W^\ast\mathcal P\otimes^{\mathbf L}(p_1^W)^\ast\mathcal O_Z$, and the $S_2$-linearization is by permuting the factors.
        By Lemma \ref{relative_fm_bc}, $\mathcal R=(i_{Z\times_WX_W})_\ast L$, where $L$ here is considered as a sheaf on $Z\times_WX_W$ via the second projection.

        By the preceding proposition, we have
        \[(r_1,r_1)^\ast\mathcal R\otimes^{\mathbf L}(r_2,r_2)^\ast\mathcal R\cong (i_{\operatorname{BKR}_Z\times_{\operatorname{BKR}_W}(X^2)_{\operatorname{BKR}_W}})_\ast(r_1^\ast L\otimes r_2^\ast L),\]
        ignoring the linearization.

        This isomorphism is $S_2$-equivariant with the latter linearized by permuting the factors.
        Indeed, the isomorphism shows that the tensor product here is acyclic, and the isomorphism itself arises from maps that are symmetric on each factor.
        Pushing this forward to $(X^2)_{\operatorname{BKR}_W}$ then gives $r_1^\ast L\otimes r_2^\ast L$.
    \end{proof}
    \begin{corollary}
        Suppose $(Z,Z')\in X^{[n],\pi}_U\times_UX^{[n],\pi,\sharp}_U$ and let $W=\pi^{[n]}(Z)=\pi^{[n]}(Z')$.
        Write $i:(X^n)_{\operatorname{BKR}_W}\to X^n$, and denote $\mathcal F^{\boxtimes n}(\operatorname{BKR}(k(Z)))=i_\ast \mathcal L$, $\mathcal F^{\boxtimes n}(\operatorname{BKR}(k(Z')))=i_\ast \mathcal L'$.

        Then $\mathcal L$ is locally $S_n$-isomorphic to the structure sheaf, and
        \[i_\ast(\mathcal L\otimes\mathcal L')=\mathcal F^{\boxtimes n}(\operatorname{BKR}(k(\mu(Z,Z')))).\]
    \end{corollary}
    \begin{proof}
        Case 1: $W$ is reduced.
        Take any $Z''=x_1\sqcup\cdots\sqcup x_n$ over $W$.
        Note that $p_i=\pi(x_i)$ are all distinct.
        Then by Lemma \ref{induction_FM} and Lemma \ref{closed_embeddings_outer}, we have
        \begin{align*}
            \mathcal F^{\boxtimes n}(\operatorname{BKR}(k(Z'')))&=\operatorname{ind}_1^{S_n}(\mathcal F(k(x_1))\boxtimes^{\mathbf L}\cdots\boxtimes^{\mathbf L}\mathcal F(k(x_n)))\\
            &=\operatorname{ind}_1^{S_n}i_{X_{p_1}\times\cdots\times X_{p_n}}(L_1\boxtimes\cdots\boxtimes L_n),
        \end{align*}
        where $\mathcal F(k(x_i))=(i_{X_{p_i}})_\ast L_i$ (these outer tensors are acyclic due to Proposition \ref{mcm_outer_tensor}).
        The claims readily follow.

        Case 2: $W$ is non-reduced.
        Take $Z''=Z_2\sqcup x_3\sqcup\cdots\sqcup x_n\in X^{[n],\pi,\sharp}$ over $W$, where $\pi^{[2]}(Z_2)_{\rm red}$ is a point $p$ such that $p,p_i=\pi(x_i)$ are all distinct.
        Let $(i_{X_{p_i}})_\ast L_i=\mathcal F(k(x_i))$.

        Suppose $Z_2$ is non-reduced.
        Then $Z_2\in X^{[2],\pi}$.
        Write $W_2=\pi^{[2]}(Z_2)$.
        Then by the same argument as in case 1 and the preceding corollary,
        \[\mathcal F^{\boxtimes n}(\operatorname{BKR}(k(Z'')))=\operatorname{ind}_{S_2}^{S_n}i_{(X^2)_{\operatorname{BKR}_{W_2}}\times X_{p_3}\times\cdots\times X_{p_n}}((r_1^\ast L\otimes r_2^\ast L)\boxtimes L_3\boxtimes\cdots\boxtimes L_n),\]
        where $L,r_i$ are as in the preceding corollary.

        Suppose $Z_2$ is reduced, say $Z_2=x_+\sqcup x_-$ with $\mathcal F(k(x_{\pm}))=(i_{X_p})_\ast L_\pm$, then by the same argument as in case 1,
        \begin{align*}
            \mathcal F^{\boxtimes n}(\operatorname{BKR}(k(Z'')))&=\operatorname{ind}_1^{S_n}i_{X_p^2\times X_{p_3}\times\cdots\times X_{p_n}}(L_+\boxtimes L_-\boxtimes L_3\boxtimes\cdots\boxtimes L_n).
        \end{align*}
        To see that the identity is satisfied, simply note that on $(X^2)_{\operatorname{BKR}_W}$, we have $(r_1^\ast L\otimes r_2^\ast L)\otimes\operatorname{ind}_1^{S_2}(i_{X_p^2})_\ast (L_+\boxtimes L_-)=\operatorname{ind}_1^{S_2}(\operatorname{res}_1^{S_2}(r_1^\ast L\otimes r_2^\ast L)\otimes (i_{X_p^2})_\ast (L_+\boxtimes L_-))=\operatorname{ind}_1^{S_2}(i_{X_p^2})_\ast((L_+\otimes L_{\rm red})\boxtimes (L_-\otimes L_{\rm red}))$.
    \end{proof}
    \begin{proposition}
        The condition in Corollary \ref{fibercheck_tots} holds.
    \end{proposition}
    \begin{proof}
        We pass the result of the preceding corollary through $\operatorname{BKR}^\top$.
        Let $W=\pi^{[n]}(Z)=\pi^{[n]}(Z')$.
        By Lemma \ref{relative_fm_bc}, we may consider the diagram
        \[\begin{tikzcd}
            X^{[n]}_W\dar&{[}(\tilde{X}_n)_{\operatorname{BKR}_W}/S_n{]}\rar{p}\lar[swap]{q}\dar&{[}(X^n)_{\operatorname{BKR}_W}/S_n{]}\dar\\
            \{W\}&{[}\operatorname{BKR}_W/S_n{]}\rar\lar&{[}\operatorname{BKR}_W/S_n{]},
        \end{tikzcd}\]
        where $p,q$ are the restrictions of the respective maps on total spaces.
        We then have $\operatorname{BKR}^\top\circ (i_{(X^n)_{[\operatorname{BKR}_W/S_n]}})_\ast=(i_{X^{[n]}_W})_\ast q_\ast\mathbf Lp^\ast$.

        Let $\mathcal L,\mathcal L'$ be as in the preceding corollary.
        Since $\mathbf Lp^\ast\mathcal L=p^\ast\mathcal L$ is locally $S_n$-isomorphic to the structure sheaf, it is the pullback of an invertible sheaf $L=q_\ast\mathcal L$ on the quotient via $q$.

        Incidentally, we also have $\mathbf Lp^\ast\mathcal L'=p^\ast\mathcal L'$ by \cite[Lemma 2.3]{arinkin}.
        We will not formally use this fact.

        We can then compute
        \begin{align*}
            q_\ast\mathbf Lp^\ast(\mathcal L\otimes\mathcal L')&=q_\ast(p^\ast\mathcal L'\otimes q^\ast L)=(q_\ast p^\ast\mathcal L')\otimes L=(q_\ast p^\ast\mathcal L')\otimes(q_\ast p^\ast\mathcal L),
        \end{align*}
        which shows the result when combined with the preceding corollary.
    \end{proof}
    \section{Tate-Shafarevich twists}\label{ss:tatesha}
    In this section, we briefly explain how the constructions interact with Tate-Shafarevich twists.
    We do this in order to access elliptic surfaces without a section (but still with integral fibers).

    We work in the analytic topology.
    \begin{definition}
        Suppose $M\to B$ is proper flat.
        Let $\mathfrak G$ be a connected group over $B$ acting on $M\to B$.
        For any covering $\mathcal U$ of $B$ and $\sigma\in\check{H}^1_{\mathcal U}(\mathfrak G)$, the Tate-Shafarevich twist $M^\sigma\to B$ is obtained by gluing $\{M_U:U\in\mathcal U\}$ along $\sigma_{U,U'}:M_{U\cap U'}\to M_{U\cap U'}$.
    \end{definition}
    \begin{definition}
        An open cover $\mathcal U$ of $C$ is $n$-regular if every $n$ points of $C$ is contained in some $U\in\mathcal U$.
    \end{definition}
    \begin{lemma}
        Let $Y\to C$ be an elliptic surface with integral fibers and let $X=\overline{\operatorname{Pic}}_{Y/C}^0\to C$ be its relative compactified Jacobian.
        Then for any $n$, there is an $n$-regular cover $\mathcal U$ of $C$ such that $Y=X^\sigma$ for some $\sigma\in\check H^1_{\mathcal U}(X^\circ)$.
    \end{lemma}
    \begin{proof}
        Since we are working in the analytic topology, we can choose some $n$-regular cover $\mathcal U$ such that for all $U\in\mathcal U$, $Y_U\to U$ has a section $s_U$, necessarily through its smooth locus.
        Let $X=\overline{\operatorname{Pic}}_{Y/C}^0$ and let $s$ be its zero-section.
        We have isomorphisms $\psi_U:Y_U\to X_U$ sending $x\in Y_U$ (say over $p\in U$) to the point corresponding to $\mathcal O_{Y_p}(x-s_U)$.
        We can then set $\sigma_{U,V}=\psi_U|_{U\cap V}(s_V|_{U\cap V})$.
    \end{proof}
    Let us now fix $X,Y,\sigma,\mathcal U$ as in the lemma.
    Since $\mathcal U$ is $n$-regular, $\mathcal U^{[n]}=\{U^{[n]}:U\in\mathcal U\}$ is an open cover of $C^{[n]}$.
    Note also that $U^{[n]}\cap V^{[n]}=(U\cap V)^{[n]}$.

    For $U,V\in\mathcal U$, $\sigma_{U,V}$ is a section of $X_{U\cap V}^\circ\to U\cap V$.
    We set $\Sigma_{U^{[n]}, V^{[n]}}=(\sigma_{U,V})^{[n]}$, which is now a section of $X^{[n],\pi}_{U^{[n]}\cap V^{[n]}}\to U^{[n]}\cap V^{[n]}$.
    \begin{proposition}
        The cochain $\Sigma$ of $X^{[n],\pi}$ is a cocycle, and $Y^{[n]}=(X^{[n]})^{\Sigma}$.
    \end{proposition}
    \begin{proof}
        The group structure on $X^{[n],\pi}$ allows $\Sigma_{U^{[n]}, V^{[n]}}$ to define an automorphism of $X^{[n],\pi}_{U^{[n]}\cap V^{[n]}}=(X_{U\cap V})^{[n]}$.
        On the other hand, the group structure on $X^\circ$ allows $\sigma_{U,V}$ to define an automorphism of $X_{U\cap V}$, hence induces an automorphism of $(X_{U\cap V})^{[n]}$.

        To show both claims, it suffices to check that these two automorphisms coincide.
        Since everything involved is separated, this can be checked on an open set, which we take to be the locus of reduced length $n$ subschemes contained in $X^\circ$ such that no two belong to the same fiber of $\pi$.
        The two automorphisms clearly agree here.
    \end{proof}
    This then identifies the Hilbert scheme of any elliptic surface $Y\to C$ with integral fibers as a Tate-Shafarevich twist of the Hilbert scheme of its Jacobian $X=\overline{\operatorname{Pic}}_{Y/C}^0\to C$, with the cocycle taking values in $X^{[n],\pi}$.

    Let us now see how the kernel $\mathcal P_n^{\rm BKR}$ behaves under such a twist.
    For $U\in\mathcal U$, we denote by $\mathcal Q_U$ the restriction of $\mathcal P_n^{\rm BKR}$ to $X_{U^{[n]}}\times_{U^{[n]}}X_{U^{[n]}}$.
    \begin{proposition}
        Under the Tate-Shafarevich twist by $\Sigma$, $\mathcal Q_U$ glues to an $(1\boxtimes\alpha)$-twisted sheaf on $Y^{[n]}\times_{C^{[n]}}X^{[n]}$ for some $\alpha\in\operatorname{Br}(X^{[n]})$.
        It establishes an exact equivalence $D^b\operatorname{Coh}(Y^{[n]})\cong D^b\operatorname{Coh}(X^{[n]},\alpha)$.
    \end{proposition}
    \begin{proof}
        This is essentially a formal consequence of theorem of the square.
        Indeed, for any $U,V\in\mathcal U$, we have the identity
        \[(\Sigma_{U^{[n]},V^{[n]}}\times\operatorname{id})^\ast\mathcal P_n^{\rm BKR}|_{U^{[n]}\cap V^{[n]}}=\operatorname{pr}_2^\ast L_{U,V}\otimes\mathcal P_n^{\rm BKR}|_{U^{[n]}\cap V^{[n]}}\]
        by Proposition \ref{tots}, where $L_{U,V}$ is a line bundle on $X_{U^{[n]}\cap V^{[n]}}^{[n]}$ given by restricting $\mathcal P_n^{\rm BKR}$ to
        \[\Sigma_{U^{[n]},V^{[n]}}\times_{U^{[n]}\cap V^{[n]}}X^{[n]}_{U^{[n]}\cap V^{[n]}}\cong X^{[n]}_{U^{[n]}\cap V^{[n]}}.\]

        Since $\Sigma$ is a cocycle, $L_{U,V}\otimes L_{V,W}\cong L_{U,W}$ on the triple intersection (by further restricting the identity, for example).
        This collection of line bundles then defines a gerbe, and thus a Brauer class $\alpha\in\operatorname{Br}(X^{[n]})$ (see \cite{caldararu}), and the identity means that the collection $\mathcal Q_U$ becomes a $(1\boxtimes\alpha)$-twisted sheaf on $Y^{[n]}\times_{C^{[n]}}X^{[n]}$.

        Note, finally, that the canonical bundle of $X^{[n]}$ is a line bundle pulled back from $X^{(n)}$, which is a line bundle pulled back from $C^{(n)}=C^{[n]}$ by the canonical bundle formula.
        Hence the kernel establishes an equivalence by \cite[Theorem 3.2.1]{caldararu}.
    \end{proof}
    \bibliographystyle{plain}
    \bibliography{ref.bib}
\end{document}